\documentclass[final,3p,times]{elsarticle}

\usepackage{amssymb}
\usepackage{amsthm}
\usepackage{amssymb,amsmath,graphicx,epsfig,psfrag}
\usepackage{slashed}

\def\ds{\displaystyle}

\newcommand{\bb}{{{\bf{  b}}}}
\newcommand{\bw}{{{\bf{  w}}}}

\newtheorem{lemma}{Lemma}

\begin{document}

\begin{frontmatter}



\title{ALE-SUPG finite element method for convection-diffusion problems in time-dependent domains: Conservative form}


\author{Sashikumaar~Ganesan\corref{cor1}}
\ead{sashi@serc.iisc.in}

\author{Shweta~Srivastava\fnref{fn2}}
\ead{shweta@nmsc.serc.iisc.ernet.in}


\cortext[cor1]{Corresponding author}

\fntext[fn2]{This work is partially supported by the National Mathematics Initiative (NMI) at Indian Institute of Science, Bangalore and by the CSIR, India}

\address{Computational Science, SERC, Indian Institute of Science, Bangalore  560012, India.}

 

\begin{abstract}
A Streamline Upwind Petrov-Galerkin (SUPG) finite element method for transient convection-diffusion-reaction equation in time-dependent domains is proposed. In particular, a convection dominated transient scalar problem is considered. The time-dependent domain is handled by the arbitrary Lagrangian-Eulerian (ALE) approach, whereas the SUPG finite element method is used for the spatial discretization. Further, the first order backward Euler and the second order Crank-Nicolson methods are used for the temporal discretization. It is shown that the stability of the semidiscrete (continuous in time) conservative ALE-SUPG equation is independent of the mesh velocity, whereas the stability of the fully discrete problem is unconditionally stable for implicit Euler method and is only conditionally stable for Crank-Nicolson time discretization. Numerical results are presented to show the influence of the SUPG stabilization parameter in a time-dependent domain. Further, the proposed numerical scheme is applied to a boundary/layer problem in a time-dependent domain.
\end{abstract}


\end{frontmatter}


\section{Introduction}
This paper deals with the numerical approximation of a transient convection-diffusion equation in time-dependent domains. It is well known that the standard Galerkin approach consists spurious oscillations in the numerical solution of a convection dominated equation. Therefore, stabilized methods have been used for convection dominated problems to suppress the oscillations, and to enhance the stability of the numerical solution. However, the influence of the convection term depends on the choice of the approach that we use to handle the domain movement. In the Eulerian approach, a fixed mesh is used for solving the equations in time dependent domains, and the moving boundaries/interfaces are captured using interface capturing/tracking methods such as Level-set~\cite{JAS96}, Volume of Fluid~\cite{HIR74}, Immersed boundary/Front-tracking method~\cite{PES02,TRY01}, etc.

Alternatively, the arbitrary Lagrangian-Eulerian (ALE) approach~\cite{DON83,BO04} has also been used when the application demands higher accuracy and/or sharp moving boundaries/interfaces, for instance, in fluid-structure interaction applications. The ALE approach introduces a convective mesh velocity term into the  model equation, and it alters the overall convective field of the problem~\cite{NOB01}. 
 However, the mesh velocity need not be identical or even close to the convective velocity in many practical applications. Therefore, the model problem can still be convection dominated, and can have boundary/interior layers even after reformulating the model equations into an ALE form. It is of our interest in this paper. 

 Streamline Upwind Petrov-Galerkin (SUPG) is one of the popular stabilization method for convection dominated problems~\cite{BH82,BUR10,JOH10,GANW10} in fixed domains.
Other popular stabilization methods such as  Galerkin least-squares~\cite{hughes89}, edge stabilization~\cite{BUR04}, continuous interior penalty~\cite{BUR06}, local projection stabilization~\cite{GAN10}, orthogonal sub-grid scale~\cite{codina2000} have also been  proposed in the literature for fixed domains, see~\cite{ROSS08} for an overview. A comparison of the SUPG method with other stabilization methods for a problem in fixed domain can be found in~\cite{codina98}. An adaptive SUPG method for a transient problem in fixed domain has been analyzed in~\cite{volker014}. In ~\cite{volkerSchmeyer08}, a comparative study of different SUPG stabilization parameters has been done. However, to the best of the authors knowledge, the SUPG has not been studied for equations in time dependent domains.
Nevertheless, an analysis of the orthogonal sub-grid scale method with the ALE approach for the solution of the  transient convection-diffusion equation in a time-dependent domain has been presented in~\cite{codina06}.
 Recently, a higher order discontinuous Galerkin (dG) method in time for convection-diffusion equation in deformable domains, with ALE framework to handle the domain movement has been proposed in ~\cite{Bon013,BonKyza013}.

In this work, we analyze the SUPG finite element method for a convection dominated transient convection-diffusion equation in a time-dependent domain. We first obtain the conservative ALE formulation for the transient equation, and then apply the SUPG discretization in space. We analyze two different, (i) the first order backward Euler and, (ii) the second order Crank-Nicolson time discretizations for the inconsistent SUPG form.  

The paper is organized as follows. In section 2, the transient convection-diffusion equation in a time-dependent domain and its ALE formulation are given. The spatial discretization using the SUPG finite element method is also presented in this section. Further, the stability of the semidiscrete problem (continuous in time) is derived in section 3. Section~4 is devoted to the stability estimates of the fully discrete problem obtained with backward Euler and Crank-Nicolson time discretization. Stability estimates for the conservative ALE SUPG form is derived in this section. Finally, the numerical studies are presented in Section~5.

\section{Model problem and its ALE formulation}
Let $ \rm{T}$ be a given time, and $t\in[0,\rm{T}]$. We consider a linear time-dependent convection-diffusion-reaction equation  
\begin{equation}\label{model}
\begin{array}{rcll}
\ds\frac{\partial u}{\partial t} - \epsilon\Delta u +  \mathbf b \cdot\nabla u + cu &=& f  &\qquad \text{ in} \,\ (0,\rm{T}] \times \Omega_t,\\
                 u &=& 0 &\qquad   \text{ on} \,\ [0,\rm{T}] \times \partial\Omega_t,\\
                 u(0,x) &=& u_{0}(x) & \qquad  \text{  in} \,\ \Omega_0, 
\end{array}
\end{equation}
where $\Omega_t\subset R^{d},~d=1,2,3$ is a time-dependent deforming/moving domain with the time dependent boundary $\partial\Omega_t$. Here,  $u(t,x)$ is an unknown scalar function, $\epsilon$ is a constant diffusion coefficient, $ \mathbf b(t,x)$ is a given convective velocity, $c(t,x)$  is a reaction function, $f(x)$ is a source term and  $u_{0}(x)$ is a  given initial data. We  assume that  $\Omega_t$ is bounded for each  $t\in[0,\rm{T}]$ with Lipschitz boundary, and there exists a constant $\mu$ such that
\begin{equation}\label{assump1}
0 < \mu  \leq   \left(c - \frac{1}{2}\nabla \cdot \mathbf{b}\right)(x), \quad \forall~x \in \Omega_t.\\
\end{equation}
Furthermore, we assume that the given data are sufficiently smooth.  
We now derive the arbitrary Lagrangian-Eulerian form of the considered model problem~\eqref{model}. Let $\hat\Omega$ be a reference domain, and define a family of bijective ALE mappings
\[
\mathcal{A}_t:\hat{\Omega} \rightarrow \Omega_t,  \qquad  \mathcal{A}_t(Y)= x(Y, t), \qquad t\in(0,\rm{T}).
\]
The reference domain  $\hat\Omega$ can simply be the initial domain $\Omega_0$ or the previous time-step domain when the deformation of the domain is large. Next, for a  function $v\in C^0({{\Omega_t}})$ on the Eulerian frame, we define their corresponding function $\hat v\in C^0({{\hat\Omega} })$ on the ALE frame as 
\[
\hat v :\hat\Omega\times (0, {\rm{T}}) \rightarrow \mathbb{R}, \qquad 
 \hat{v}:=v\circ \mathcal{A}_t,\qquad \text{with} \qquad \hat{v}(Y, t) = v(\mathcal{A}_t(Y), t).
\]
Further, the time derivative on the ALE frame is defined as 
\[
\ds\frac{\partial v}{\partial t} \Big|_{Y}:\Omega_t\times (0, {\rm{T}}) \rightarrow \mathbb{R}, \qquad 
\ds\frac{\partial v  }{\partial t} \Big|_{Y }(x,t) = \ds\frac{\partial \hat v }{\partial t}(Y,t), \qquad 
 Y= \mathcal{A}_t^{-1}(x).
 \]
We now apply the chain rule to the time derivative of $v\circ\mathcal{A}_t$ on the ALE frame to get
\[
\ds\frac{\partial v}{\partial t} \Big|_{Y} = \ds\frac{\partial v}{\partial t} (x,t)
 + \ds\frac{\partial x}{\partial t}\Big|_{Y}\cdot\nabla_x v  = \ds\frac{\partial v}{\partial t}  
 + \ds\frac{\partial \mathcal{A}_t(Y)}{\partial t}  \cdot\nabla_x v= \ds\frac{\partial v}{\partial t}  
 + \bw\cdot\nabla_x v,
\]
where $\bw$ is the domain velocity. Using this relation in the model problem~\eqref{model}, we get
\begin{equation}\label{ALEmodel}
\ds\frac{\partial u}{\partial t} \Big|_{Y} - \epsilon\Delta u +  (\mathbf b -\bw)\cdot\nabla u + cu = f.
\end{equation}
The conservative ALE form of equation is given 
\begin{equation}\label{consermodel}
{\left. \ds\frac{\partial (u J_{A_{t}})}{\partial t} \right|_Y} + J_{A_{t}} \left[ - \epsilon\Delta u +  (\bb-\bw) \cdot\nabla u + (c - \nabla \cdot \bw) u \right]  = J_{A_{t}}f  
\end{equation}
The main difference between~\eqref{model} and~\eqref{consermodel} is the additional domain velocity in the ALE form that account for the deformation of the domain, see for explanation \cite{NOB01}.
 
\subsection{Variational form of the conservative ALE equation}
To derive the variational form, let us define the functional space for the equation~\eqref{consermodel} with the ALE mapping:
\[
 V = \left\{v\in  H_{0}^{1}(\Omega_{t}),  ~~v:\Omega_t\times (0,T] \rightarrow \mathbb{R}, ~~ v=\hat{v}  \circ A_{t}^{-1}, ~~ \hat{v} \in H_{0}^{1}(\hat{\Omega}) \right\}.
\]
Now, multiplying the equation~\eqref{consermodel} with a test function $v\in V$, integrate over $\Omega_t$, and after applying integration by parts to the higher order derivative term, the variational form of the equation~\eqref{consermodel} reads:\\

\noindent For given $\mathbf b$, $\bw$, $c$, $u_0$ and $f$, find $u\in V$ such that for all $ t\in(0,T]$
\begin{equation}\label{weakALE}
\frac{d }{dt}\left( u,~v\right) + (\epsilon\nabla u, ~\nabla v) + \left(  (\mathbf b -\bw) \cdot \nabla u ,~v\right) + \left((cu - \nabla \cdot \bw), ~v \right) = (f,~v),  \qquad  v\in V.
\end{equation}
Here,  $(\cdot,~\cdot)$ denotes the $L^2-$inner product in $\Omega_t$. The stability analysis for the standard Galerkin finite element discretization~\eqref{weakALE} can be seen in~\cite{BO04,NOB01,MAC12}. Here, we will therefore concentrate on the SUPG discretization of the conservative ALE form~\eqref{weakALE}.

\subsection{SUPG discretization of the ALE equation}
It is well-known that the standard Galerkin finite element discretization of convection-diffusion equation induces spurious oscillations in the numerical solution in convection dominated cases. Note that the convective term in the ALE form~\eqref{weakALE} is $(\mathbf b -\bw)$, and the instabilities and spurious oscillations are not expected when the mesh velocity is same as the convective velocity (pure Lagrangian form). However, this is not the case in the ALE form, and in practice, the mesh velocity need not be in the same direction as the convective velocity. Therefore,  to circumvent the instabilities and to suppress the spurious oscillations, a stabilization method has to be used in practical applications, in particular, for problems with boundary and interior layers. One of the simple and most popular stabilization method for convection dominated problems in fixed domains is the SUPG method, and is considered here. 

Let $\mathcal{T}_{h,t}$ be the collection of simplices obtained by triangulating  the time-dependent domain $\Omega_t$. We denote the diameter of the cell $K  \in\mathcal{T}_{h,t}$    by $h_{K,t}$ and  the global mesh size in the triangulated  domain $\Omega_{h,t}$   by $h_t:=\max\{h_{K,t}~:~ K \in\mathcal{T}_{h,t}\}$.   Suppose $V_{h}\subset V$  is a  conforming finite element (finite dimensional) space. Let $\phi_h:= \{\phi_i(x)\} $, $i=1,2,...,\mathcal{N},$ be the finite element basis functions of $V_{h}$. The discrete finite element space $V_{h}$ is then defined by  
\[
  V_{h}  = 
\left\{u_h:~u_h(t,x)= \sum_{i=1}^{\mathcal{N}}u_{i}(t)
\phi_i(x) ;\quad u_{i}\in\mathbb{R}\right\}\subset
H_0^{1}(\Omega_t).
\]
We next define the discrete ALE mapping $\mathcal{A}_{h,t}(Y)$ and the discrete mesh velocity $\bw_h$ in space. We use the piecewise linear Lagrangian finite element space
\[
 \mathcal{L}^{1}(\hat\Omega) = \left\{  \psi\in H^1(\hat \Omega):  \psi|_K\in{\rm P_1}(\hat K)  \text{ for all } \hat K\in  \hat \Omega_h \right\},
\]
where ${\rm P_1}$ is a set of polynomials of degree less than or equal to one  on $\hat K$. Using this linear space, we define the semidiscrete ALE mapping in space for each $t\in[0,{\rm{T})}$ by
\begin{equation} \label{discALE}
 \mathcal{A}_{h,t}:\hat{\Omega_h} \rightarrow \Omega_{h,t}.
\end{equation}
Further, the discrete (continuous in time) mesh velocity $\bw_h(t,Y)\in \mathcal{L}^{1}(\hat\Omega)^d$ in the   ALE frame for each $t\in[0,{\rm{T})}$ is defined   by
\[
 \hat\bw_h(t,Y)=   \sum_{i=1}^{\mathcal{M}}\bw_{i}(t)
\psi_i(Y) ;\quad \bw_{i}(t)\in\mathbb{R}^d. 
\]
Here, $\bw_{i}(t)$ denotes the mesh velocity of the $i^{th}$ node of simplices at time $t$ and $\psi_i(Y) $, $i=1,2,...,\mathcal{M},$ are the  basis functions of $\mathcal{L}^{1}(\hat\Omega)$. We then define the semidiscrete mesh velocity in the Eulerian frame as 
\[
  \bw_h(t,x) =  \hat\bw_h  \circ\mathcal{A}^{-1}_{h,  t}(x).
\]
Using the above finite element spaces and applying the inconsistent SUPG finite element discretization to the ALE form~\eqref{weakALE}, the  semi-discrete form in space
of~\eqref{weakALE} reads:\\

\noindent for given $u_h(0)=u_{h,0}$,  $\mathbf b$, $\bw_h$, $c$, $f$, and $\Omega_0$, find $ u_h(t,x)\in   V_{h}$ such that for all $ t\in(0,T]$

\begin{equation} \label{semidisc_conser}
\begin{aligned} 
\frac{d}{dt} \left(u_h , v_h\right)    & + a_{SUPG}(u_h,v_h)    -  \int_{\Omega_{h,t}} \nabla \cdot( \bw_h  u_h)~ v_h~  dx \\ 
   & = \int_{\Omega_{h,t}} fv_h~   dx 
+ \sum_{K  \in \mathcal{T}_{h,t}} \delta_{K}  \int_{K} f ~(\mathbf{b-w}_h) \cdot\nabla v_h ~ dK
\end{aligned}
\end{equation}
where
\begin{align}\label{supg}
 a_{SUPG}(u,v)  &= \epsilon(\nabla u, \nabla v) +  (\mathbf{b} \cdot \nabla u, v) + (c u, v) \nonumber\\
 &+\sum_{K \in \mathcal{T}_{h,t}} \delta_{K} (- \epsilon\Delta u + (\mathbf{b-w}_h) \cdot\nabla u + cu, (\mathbf{b-w}_h ) \cdot \nabla v)_K
\end{align}
Here,  $(\cdot,~\cdot)$ denotes the $L^2-$inner product in   $\Omega_{h,t}$ and  $\delta_{K}$ is a local stabilization parameter. Further,   $u_{h,0}\in V_{h}$ is defined as the $L^2$-projection of the initial value $u_0$ onto $V_{h}$. 
 
\begin{lemma} \label{lemma1}
Coercivity of $a_{SUPG}(\cdot , \cdot)$: Let the discrete form of the assumptions~\eqref{assump1} be satisfied. Further,  assume that the SUPG parameters satisfy
\begin{equation}\label{assump2}
 \delta_{K} \le  \frac{\mu_{0}}{2||c||_{K,\infty}^{2}}, \qquad    \delta_{K} \le  \frac{h_{K}^{2}}{2\epsilon c_{inv}^{2}},
\end{equation}
where $c_{inv}$ is a constant used in inverse inequality. Then, the SUPG bilinear form  satisfies
\[
 a_{SUPG}(u_h,u_h) \geq \frac{1}{2}|||u_h|||^{2},
\]
where the mesh dependent norm is defined as
\[
 |||u|||^{2} = \left(\epsilon| u|^{2}_{1}  + \sum_{K \in \mathcal{T}_{h,t}} \delta_{K} ||(\mathbf{b-w}_h) \cdot \nabla u||^{2}_{0,K} + \mu||u||_{0}^{2}\right).
\]
\end{lemma}
\begin{proof}
 Using the assumption~\eqref{assump1} in~\eqref{supg}, we get
\begin{align}
 a_{SUPG}(u_h,u_h) \ge \epsilon|u_h|^{2}_{1}  & + \mu||u_h||_{0}^{2}   + \sum_{K \in \mathcal{T}_{h}} \delta_{K} ||(\mathbf{b-w}_h) \cdot \nabla u_h||^{2}_{0,K} \nonumber \\
\label{lemeq1}
  & + \sum_{K \in \mathcal{T}_{h,t}} \delta_{K} (- \epsilon \Delta u_h +  cu_h, (\mathbf{b-w}_h) \cdot \nabla u_h)_{K}
\end{align}
Considering the last term in the above inequality, we have
\begin{align}
  &\left|\sum_{K \in \mathcal{T}_{h,t}}\delta_{K} (- \epsilon\Delta u_h +  cu_h,  (\mathbf{b-w}_h) \cdot \nabla u_h)_{K}\right| \nonumber \\
&\qquad \leq \sum_{K \in \mathcal{T}_{h,t}} \left[\epsilon^{2}\delta_{K}||\Delta u_h||^{2}_{0,K} + \delta_{K} c^2||u_h ||^{2}_{0,K} +\frac{1}{2}\delta_{K}\left||(\mathbf{b-w}_h) \cdot \nabla u_h\right||^{2}_{0,K} \right] \nonumber \\
 &\qquad \leq \frac{1}{2}\left[\epsilon |u_h|_{1,K}^{2} + \mu ||u_h||_{0,K}^{2} + \delta_{K}||(\mathbf{b-w_h})\cdot \nabla u_h||_{0,K}^{2}\right] \nonumber \\
\label{lemeq2}
&\qquad \leq \frac{1}{2} |||u_h|||^{2} 
\end{align}
Here, the inverse inequality
\[
  || \Delta u_h ||_{0,K}  =  c_{inv} h_{K}^{-1}|u_h|_{1,K}, \,\ \forall\,\ u_h  \in V_h.
\]
has been used in the diffusive term. Note that the inverse inequality and the second assumption on $\delta_k$ in~\eqref{assump2} can be omitted when piecewise linear finite elements are used. Using the estimate~\eqref{lemeq2} in~\eqref{lemeq1}, the coercivity is proved. 
\end{proof}
\section{Stability of the semidiscrete (continuous in time) ALE-SUPG problem in space} 
\subsection{Stability of the semidiscrete (continuous in time) conservative ALE-SUPG form}
In the case of conservative form \eqref{consermodel}, we can not take $\psi_h = u_h$, for the stability of semi discrete scheme, since two functions can have different time evolution. We can express $u_h$ as a linear combination of test functions with time dependent unknown coefficients as
\begin{equation*}
 u_h(x,t) = \sum_{i \in N}u_i(t) \psi_i(x,t) 
\end{equation*}
Since the functions in reference domain doesn't depends on time, 
\begin{equation*}
 \left.\frac{\partial u_h}{\partial t} \right \arrowvert_{Y}(x,t) = \sum_{i \in N}\psi_i(x,t) \frac{du_i}{dt}(t).
\end{equation*}
The finite element semi discrete approximation of the equation reads as,
\begin{eqnarray*} 
&\ds \frac{d}{dt} \int_{ {\Omega_t}} u_h  \psi_h ~dX +  \int_{ {\Omega_t}} \epsilon\nabla u_h \cdot\nabla  \psi_h~ dX +  \int_{ {\Omega_t}}  \bb \cdot\nabla u_h ~\psi_h~  dX + \int_{ {\Omega_t}} c  u_h~ \psi_h~  dX \nonumber\\
&\ds+ \sum_{K \in \mathcal{T}_{h}} \delta_{K} \left( - \epsilon\Delta u_h + (\mathbf{b-w_h}) \cdot\nabla u_h + cu_h, (\mathbf{b-w_h})\cdot\nabla \psi_h \right)- \int_{ {\Omega_t}}  \nabla \cdot (\bw_h ~ u_h )~\psi_h~  dX\nonumber \\
 &\ds = \int_{ {\Omega_t}} f\psi_h~  dX + \sum_{K \in \mathcal{T}_{h}} \delta_{K} (f, (\mathbf{b-w_h})\cdot\nabla \psi_h)_{K}
\end{eqnarray*}
Taking $\psi_h = \psi_i$ and multiplying the equation by $u_i(t)$ we get,
\begin{align*}
& \ds u_i(t)\frac{d}{dt} \int_{ {\Omega_t}} u_h  \psi_i ~dX +  \int_{ {\Omega_t}} \epsilon\nabla u_h \cdot\nabla(u_i(t)  \psi_i)~ dX +  \int_{ {\Omega_t}}  \bb \cdot\nabla u_h ~u_i(t)\psi_i~  dX + \int_{ {\Omega_t}} c  u_h~ u_i(t)\psi_i~  dX\nonumber\\
&+ \sum_{K \in \mathcal{T}_{h}} \delta_{K} \Big( - \epsilon\Delta u_h + (\mathbf{b-w_h}) \cdot\nabla u_h + cu_h, (\mathbf{b-w_h})\cdot\nabla u_i(t) \psi_i \Big)- \int_{ {\Omega_t}}  \nabla \cdot (\bw_h ~ u_h )~u_i(t)\psi_i~  dX \nonumber\\
 & = \int_{ {\Omega_t}} f~u_i(t)\psi_i~  dX + \sum_{K \in \mathcal{T}_{h}} \delta_{K} (f, (\mathbf{b-w_h})\cdot\nabla (u_i(t)\psi_i))_{K}
\end{align*}
The first term can be written as,
\begin{equation*} 
\begin{aligned} 
  \ds u_i(t)\frac{d}{dt} \int_{ {\Omega_t}} u_h  \psi_i ~dX& = \frac{d}{dt}\int_{ {\Omega_t}} u_h u_i(t) \psi_i dX - \int_{ {\Omega_t}} u_h \psi_i \frac{du_i(t)}{dt}~dX\nonumber\\
&=\frac{d}{dt}\int_{ {\Omega_t}} u_h u_i(t) \psi_i dX - \int_{ {\Omega_t}} u_h \left.\frac{\partial \psi_i u_i(t)}{\partial t}\right|_Y~dX
\end{aligned}
\end{equation*}
Summing over $i$ to get 
\begin{align*}
 \frac{d}{dt} ||u_h||^2_{L_2(\Omega_t)}& - \int_{ {\Omega_t}} u_h \left.\frac{\partial u_h}{\partial t}\right|_Y~dX+  \epsilon||\nabla u_h ||^2_{L_2(\Omega_t)} +  \int_{ {\Omega_t}}  \bb \cdot\nabla u_h ~u_h~  dX - \int_{ {\Omega_t}}  \nabla \cdot (\bw_h ~ u_h )~u_h~  dX\nonumber\\
&+ c  ||u_h||^2_{L_2(\Omega_t)}  dX+ \sum_{K \in \mathcal{T}_{h}} \delta_{K} \Big( - \epsilon\Delta u_h + (\mathbf{b-w_h}) \cdot\nabla u_h + cu_h, (\mathbf{b-w_h})\cdot\nabla u_h \Big)\nonumber\\
 & = \int_{ {\Omega_t}} f~u_h~  dX + \sum_{K \in \mathcal{T}_{h}} \delta_{K} (f, (\mathbf{b-w_h})\cdot\nabla u_h)_{K},
\end{align*}
where the relations
\begin{equation*}
 \int_{\Omega_{h,t}} \left.\frac{\partial u_h}{\partial t} \right \arrowvert_{Y}u_h~ d x = \frac{1}{2}\left( \frac{d}{dt}||u_h||_{0}^2 - \int_{\Omega_{h,t}} u_h^2 \nabla \cdot \mathbf{w_h} d x\right)
\end{equation*}
is used.
Next, applying the Cauchy-Schwartz and Young’s inequalities to the right hand side terms to get
\begin{eqnarray*}
  | (f,u_h )| = \left(\frac{f}{\mu^{1/2}}, \mu^{1/2}u_h\right)   
\leq \frac{1}{\mu}||f||_{0}^{2} +  \frac{1}{4}\mu||  u_h||_{0}^{2}
\end{eqnarray*}
and
\begin{eqnarray*}
 \left| \sum_{K \in \mathcal{T}_{h,t}}  \delta_{K} (f,(\mathbf{b-w}_h) \cdot \nabla u_h)_{K}\right| \leq \sum_{K \in \mathcal{T}_{h,t}}\delta_{K} ||f||_{0}^{2} + \frac{1}{4}\sum_{K \in \mathcal{T}_{h,t}}\delta_{K}||(\mathbf{b-w}_h) \cdot \nabla u_h)||_{0,K}^{2} . 
\end{eqnarray*}
Hence, we obtain
\begin{align*}
 \frac{d}{dt}||u_h||^2_{0} + |||u_h|||^2  
&\le \frac{2}{\mu}||f||_{0}^{2} +  \frac{1}{2}||\mu^{1/2}u_h||_{0}^{2} + 2\sum_{K \in \mathcal{T}_{h,t}}\delta_{K} ||f||_{0}^{2}+ \frac{1}{2}\sum_{K \in \mathcal{T}_{h,t}}\delta_{K}||(\mathbf{b-w}_h) \cdot \nabla u_h)||_{0,K}^{2} \\
&\le \frac{2}{\mu}||f||_{0}^{2}  + 2\sum_{K \in \mathcal{T}_{h}}\delta_{K} ||f||_{0}^{2}+\frac{1}{2} |||u_h|||^{2}
\end{align*}
Finally, integrating the above equation over $(0,T)$, we get the stability estimate for conservative ALE-SUPG scheme
\[
 ||u_h||^2_{0}  +  \frac{1}{2} \int_0^T  |||u_h|||^2 dt  \leq ||u_h(0)||^2_{0} +\frac{2}{\mu} \int_0^T
 ||f||_{0}^{2} ~dt + 2\int_0^T \sum_{K \in \mathcal{T}_{h,t}}\delta_{K} ||f||_{0}^{2}~ dt,
\]
which is independent of mesh velocity field.

\section{Fully discrete scheme} In this section, we present the stability estimates for a fully discrete conservative ALE-SUPG form. In particular, the first order implicit backward Euler and the second order modified Crank-Nicolson time discretizations are analyzed.

\subsection{Discrete ALE-SUPG with Implicit Euler method}
Let $0=t^0<t^1<\dots <t^N={T}$ be a decomposition of the considered
time interval $[0,  {\rm T}]$ into $N$ equal time intervals. Let us denote the uniform time step by $\Delta t = \tau^n   $ = $t^{n}$ - $t^{n-1}$, $1\le n \le N$. Further, let   $u_h^n$ be the
approximation of $u(t^n,x)$ in $V_{h}\subset
H_0^{1}(\Omega_{t^n})$, where $\Omega_{t^n}$ is the deforming domain at time $t=t^n$. We first discretize  the ALE mapping in time using a linear interpolation. We denote the discrete ALE mapping by 
$\mathcal{A}_{h, \Delta t}$, and define it for every $\tau\in[t^n,t^{n+1}]$ by
\[
 \mathcal{A}_{h, \Delta t}(Y) = \frac{\tau-t^n}{\Delta t}\mathcal{A}_{h,t^{n+1}}(Y) +  \frac{t^{n+1}-\tau }{\Delta t}\mathcal{A}_{h,t^{n}}(Y),
\]
where $\mathcal{A}_{h,t}(Y)$ is the time continuous ALE mapping defined in~\eqref{discALE}. Since the ALE mapping is discretized in time using a linear interpolation, we obtain the discrete mesh velocity 
\[
 \hat\bw_h^{n+1}(Y)= \frac{\mathcal{A}_{h,t^{n+1}}(Y) - \mathcal{A}_{h,t^{n}}(Y)}{\Delta t}
\]
as a piecewise constant function in time. Further, we define the mesh velocity on the Eulerian frame as
\[
  \bw_h^{n+1} =  \hat\bw_h^{n+1}  \circ\mathcal{A}^{-1}_{h,\Delta t}(x).
\]
Now, applying the backward Euler time discretization to the semidiscrete problem~\eqref{weakALE}, the fully discrete form of~\eqref{weakALE} reads:\\

\noindent For given $u_h(0)=u_{h,0}$,  $\mathbf b$, $\bw^{n+1}_h$, $c$, $f^{n+1}$ and $\Omega_0$, find $ u^{n+1}_h\in   V_{h}$ in the time interval $(t^{n},t^{n+1})$ such that for all $v_h\in V_{h}$  
\begin{equation} \label{disc}
\begin{aligned} 
& \ds\frac{1}{\Delta t} \left((u^{n+1}_h,  v_h)_{\Omega_{h,t ^{n+1}}} -(u^n_h,  v_h)_{\Omega_{h,t ^{n}}}  \right) + a^{n+1}_{SUPG}(u^{n+1}_h,v_h) - \int_{\Omega_{h,t^{n+1}}} \nabla (\bw^{n+1}_h   u^{n+1}_h)~ v_h~  dx \\ 
   &\qquad = \int_{\Omega_{h,t^{n+1}}} f^{n+1} v_h~   dx 
+ \sum_{K  \in \mathcal{T}_{h,t^{n+1}}} \delta_{K}  \int_{K} f^{n+1} ~(\mathbf{b-w}^{n+1}_h) \cdot\nabla v_h ~ dK,
\end{aligned}
\end{equation}
where
\begin{align*} 
 a^{n+1}_{SUPG}(u_h,v_h)  &= \epsilon(\nabla u_h, \nabla v_h)_{\Omega_{h,t ^{n+1}}} +  (\mathbf{b} \cdot \nabla u_h, v_h)_{\Omega_{h,t ^{n+1}}} + (c u_h, v_h)_{\Omega_{h,t ^{n+1}}} \nonumber\\
 &+\sum_{K \in \mathcal{T}_{h,t^{n+1}}} \delta_{K} (- \epsilon\Delta u_h + (\mathbf{b-w}^{n+1}_h) \cdot\nabla u_h + cu_h, ~(\mathbf{b-w}^{n+1}_h ) \cdot \nabla v_h)_K
\end{align*}
\begin{lemma}(Gronwall lemma)
Let $\Delta t$, $f_0,~A_n,~ B_n,~ C_n $ be given sequences of non-negative numbers for $n \geq 0$ such that the following inequality holds
\[
 A_n + \Delta t \sum_{i=0}^{n} B_i \leq \Delta t  \sum_{i=0}^{n} \gamma_i A_i + \Delta t \sum_{i=0}^{n} C_i + f_0. 
\]
We then have
\[
 A_n + \Delta t \sum_{i=0}^{n} B_i \leq \exp \left(\Delta t \sum_{i=0}^{n} \sigma_i \gamma_i\right)\left[\Delta t \sum_{i=0}^{n} C_i + f_0\right]
\]
 where $\sigma_i = \frac{1}{1-\gamma_i \Delta t}$ and $\gamma_i \Delta t \leq 1$ for all $i=0,\ldots,n$.
\end{lemma}

\begin{lemma}\label{stabcons} ({Stability estimates for the Conservative ALE-SUPG form with implicit Euler method})
Let the discrete version of~\eqref{assump1} and the assumption~\eqref{assump2} on $\delta_K$ hold true.  Further, assume that $\delta_K \leq \frac{\Delta t}{4}$, the solution of the conservative problem satisfies
\begin{align*}
||u_h^{n+1}||^2_{L^2 \left( \Omega_{t^{N+1}} \right)}&+ \frac{\Delta t}{2}\sum_{n=0}^{N}|||u_h^{n+1}|||^2_{L^2 \left( \Omega_{t^{n+1/2}} \right)} \\
 &\leq ||u_h^{0}||^2_{L^2 \left( \Omega_{t^{0}} \right)} + \frac{2 \Delta t}{\mu}\sum_{n=0}^{N} ||f^{n+1/2}||^2_{L^2 \left( \Omega_{t^{n+1/2}} \right)} + 2 \Delta t \sum_{K \in \mathcal{T}_{h}}  \delta_{K}\sum_{n=0}^{N}||f^{n+1/2}||^2_{L^2 \left( \Omega_{t^{n+1/2}} \right)}
 \end{align*}
\end{lemma}

\begin{proof}
Here we assume a piecewise constant in time mesh velocity field and adopt a mid point time integration rule, satisfying the GCL. We will have,
\begin{align*}
 \frac{1}{\Delta t} & \int_{\Omega_{h,t^{n+1}}}u_h^{n+1}  \psi_h ~dX -  \frac{1}{\Delta t} \int_{\Omega_{h,t^{n}}}u_h^{n} \psi_h ~dX + \int_{\Omega_{h,t^{n+1/2}}} \psi_h \nabla \cdot [(\mathbf{b - w_h}) u_h^{n+1}] ~dX\\
& + \int_{\Omega_{h,t^{n+1/2}}} \epsilon \nabla u_h^{n+1} \nabla \psi_h~dX + \int_{\Omega_{h,t^{n+1/2}}}cu_h^{n+1}\psi_h~dX\\
&+\sum_{K \in \mathcal{T}_{h}} \delta_{K} \Big( - \epsilon\Delta u_h^{n+1}  + ((\mathbf{b-w_h})(t^{n+1}) \cdot \nabla u_h^{n+1})   + cu_h^{n+1} , (\mathbf{b-w_h})(t^{n+1}) \cdot \nabla \psi_h \Big) \\
& =\int_{\Omega_{h,t^{n+1/2}}} f^{n+1/2} ~\psi_h ~dX+ \sum_{K \in \mathcal{T}_{h}}  \delta_{K} \int_{\Omega_{h,t^{n+1/2}}} f^{n+1/2} ~(\mathbf{b-w_h})(t^{n+1}) \cdot \nabla \psi_h~dK  
 \end{align*}
Taking $ \psi_h = u_h^{n+1}$ and applying integration by parts to the convective term, we get
\begin{align*}
 \int_{\Omega_{h,t^{n+1}}} u_h^{n+1} &u_h^{n+1} ~dX - \int_{\Omega_{h,t^{n}}}u_h^{n}u_h^{n+1} ~dX + \Delta t ~ a_{SUPG}^{n+1/2}(u_h^{n+1}, u_h^{n+1})- \frac{\Delta t}{2}\int_{\Omega_{h,t^{n+1/2}}} \nabla \cdot \mathbf{w}_h  |u_h^{n+1}|^2dX\\
& \leq \Delta t \int_{\Omega_{h,t^{n+1/2}}} f^{n+1/2} ~ u_h^{n+1} ~dX+ \sum_{K \in \mathcal{T}_{h}}  \delta_{K} \Delta t\int_{\Omega_{h,t^{n+1/2}}} f^{n+1/2} ~(\mathbf{b-w_h}) \cdot \nabla u_h^{n+1}~dK \\
\end{align*}
Using the coercivity of the bilinear form and the Cauchy Schwarz inequality, we get
\begin{align*}
 ||u_h^{n+1}||^2_{L^2 \left( \Omega_{t^{n+1}} \right)} + \frac{\Delta t}{2}|||u_h^{n+1}|||^2_{L^2 \left(\Omega_{t^{n+1/2}}\right)}-& \frac{\Delta t}{2}\int_{\Omega_{h,t^{n+1/2}}} \nabla \cdot \mathbf{w}_h  |u_h^{n+1}|^2dX \\& \leq  \int_{\Omega_{h,t^{n}}}u_h^{n}u_h^{n+1} ~dX + \int_{\Omega_{h,t^{n+1/2}}} \Delta t f^{n+1/2} ~ u_h^{n+1} ~dX\\
&+  \sum_{K \in \mathcal{T}_{h}}  \delta_{K} \int_{\Omega_{h,t^{n+1/2}}} \Delta t f^{n+1/2} ~(\mathbf{b-w_h}) \cdot \nabla u_h^{n+1}~dK \\
&\leq ||u_h^{n+1}||^2_{L^2 \left(\Omega_{t^{n}}\right)}+||u_h^{n}||^2_{L^2 \left(\Omega_{t^{n}}\right)}
+ \frac{\Delta t}{\mu} ||f^{n+1/2}||^2_{L^2 \left(\Omega_{t^{n+1/2}}\right)}\\
&+ \Delta t \sum_{K \in \mathcal{T}_{h}}  \delta_{K}||f^{n+1/2}||^2_{L^2 \left(\Omega_{t^{n+1/2}}\right)},
\end{align*}
where the relation 
\begin{equation*}
 ||u_h^{n+1}||^2_{L^2 \left(\Omega_{t^{n+1}}\right)} - ||u_h^{n+1}||^2_{L^2\left(\Omega_{t^{n}}\right)} = \int_{t^{n}}^{t^{n+1}} \int_{\Omega_{t}} |u_h^{n+1}|^2 \nabla \cdot \mathbf{w}_h~dX \\
=\Delta t \int_{\Omega_{h,t^{n+1/2}}} |u_h^{n+1}|^2 \nabla \cdot \mathbf{w}_h~dX 
\end{equation*}
is used. Hence, we have
 \begin{align*}
 ||u_h^{n+1}||^2_{L^2 \left(\Omega_{t^{n+1}}\right)}&+ \frac{\Delta t}{2}|||u_h^{n+1}|||^2_{L^2 \left( \Omega_{t^{n+1/2}} \right)} 
 \leq ||u_h^{n}||^2_{L^2 \left(\Omega_{t^{n}}\right)}+ 2 \Delta t \sum_{K \in \mathcal{T}_{h}}  \delta_{K}||f^{n+1/2}||^2_{L^2 \left(\Omega_{t^{n+1/2}}\right)} + \frac{2 \Delta t}{\mu} ||f^{n+1/2}||^2_{L^2 \left(\Omega_{t^{n+1/2}}\right)}
 \end{align*}
Finally, summing over all   time steps, we get the estimate.
\end{proof}

\subsection{Discrete ALE-SUPG with Crank-Nicolson method} We next consider the modified Crank-Nicolson method which is basically Runge-Kutta method of order $2$. For an equation
\begin{align*}
 \frac{du(t)}{dt} = f(u(t),t),~~~ t>0~~ and ~~u(0) = u_0
\end{align*}
 with the Crank-Nicolson, we have
\begin{align*}
 u^{n+1} - u^{n} = \Delta t f \left(\frac{u^{n+1} + u^{n}}{2}, t^{n+\frac{1}{2}} \right)
\end{align*}


\begin{lemma} (Stability estimates for the conservative ALE-SUPG form with Crank-Nicolson method)
Let the discrete version of~\eqref{assump1} and the assumption~\eqref{assump2} on $\delta_K$ hold true.  Further, assume that $\delta_K \leq \frac{\Delta t}{4}$ then
\begin{equation*}
\begin{aligned}
  \|u_h^{N+1}\|^2_{L^2(\Omega_{N+1})}  +& \frac{\Delta t}{4} \sum_{n=0}^{N}|||u_h^{n+1} + u_h^n|||^{2}_{L^2 \left(\Omega_{t^{n+1/2}} \right)} \\
&\leq  \left((1+\Delta t \beta_2^0)\| u_h^{0}\|^2_{L^2(\Omega_{0})}  + \Delta t\sum_{n=0}^{N}\left(\frac{2}{\mu}+  \Delta t \right) \|f^{n+1/2}\|^2_{L^{2}(\Omega_{t_{n+1/2}})}\right)
 \exp{\left( \Delta   t \sum_{n=1}^{N+1}\frac{\beta_1^n + \beta_2^n}{1-\Delta t(\beta_1^n + \beta_2^n)} \right)}. \\
 \end{aligned}
\end{equation*}
\end{lemma}

\begin{proof}
Applying the time discretization to the conservative SUPG-ALE equation, we get
\begin{align*}
\int_{\Omega_{h,t^{n+1}}} u^{n+1}_h  v_h~ d x - &\int_{\Omega_{h,t^{n}}} u^{n}_h  v_h ~d x 
 + \Delta t ~a^{n+1/2}_{SUPG}\left(\frac{u^{n+1}_h + u^{n}_h}{2},v_h \right) - \Delta t \int_{\Omega_{h,t^{n+1/2}}}\nabla \cdot \left(\bw^{n+1/2}_h   \left(\frac{u^{n+1}_h + u^{n}_h}{2} \right)\right)~ v_h~  dx     \\ 
&\quad  = \Delta t\int_{\Omega_{h,t^{n+1/2}}} f^{n+1/2} v_h~   dx + \sum_{K  \in \mathcal{T}_{h,t^{n+1/2}}} \delta_{K}  \int_{K}\Delta t f^{n+1/2} ~(\mathbf{b-w}_h) \cdot\nabla v_h ~ dK,
\end{align*}
 Testing the above equation with $v_h = u_h^{n+1} + u_h^n$, and using the relations
\[
(u_h, u_h + v_h) = \frac{1}{2}||u_h||^2 + \frac{1}{2}||u_h + v_h||^2 - \frac{1}{2}||v_h||^2  
\]
and
\begin{align*}
 || u_h^{n}||^2_{L^2 \left( \Omega_{t^{n+1}} \right)} = || u_h^{n}||^2_{L^2( \Omega_{t^{n}})} + \int_{t^{n}}^{t^{n+1}}\int_{\Omega_{t}} \nabla \cdot \mathbf{w}_{h}|u_h^{n}|^2~dx~ dt,
\end{align*}
 the first term can be written as,
 \begin{equation*}
\begin{aligned}
 \int_{\Omega_{h,t^{n+1}}}u_h^{n+1}(u_h^{n+1} + u_h^n)~dx~ &  - \int_{\Omega_{h,t^{n}}}u_h^{n}(u_h^{n+1} + u_h^n)~dx \\
&= \frac{1}{2} ||u_h^{n+1}||^2_{L_2(\Omega_{t^{n+1}})} + \frac{1}{2}||u_h^{n+1} + u_h^{n}||^2_{L_2(\Omega_{t^{n+1}})} - \frac{1}{2}||u_h^n||^2_{L_2(\Omega_{t^{n+1}})} - \frac{1}{2}||u_h^n||^2_{L_2(\Omega_{t^{n}})} \\&- \frac{1}{2}||u_h^{n+1} + u_h^{n}||^2_{L_2(\Omega_{t^{n}})} +  \frac{1}{2}||u_h^{n+1}||^2_{L_2(\Omega_{t^{n}})}\\
& = ||u_h^{n+1}||^2_{L_2(\Omega_{t^{n+1}})} - ||u_h^n||^2_{L_2(\Omega_{t^{n}})} + \Delta t\int_{\Omega_{t^{n+1/2}}} \nabla \cdot \mathbf{w}_{h}u_h^{n+1} u_h^{n}~dx
\end{aligned}
\end{equation*}
Applying integration by parts to the mesh velocity term, along with the coercivity of the bilinear form and using the Cauchy-Schwarz inequality to the right hand side terms, we get
 \begin{equation*}
\begin{aligned}
||u_h^{n+1}||^2_{L_2(\Omega_{t^{n+1}})} + \frac{\Delta t}{8} |||(u_h^{n+1} &+ u_h^n)|||^2_{{L_2(\Omega_{t^{n+1/2}})}}  
+  \Delta t\int_{\Omega_{h,t^{n+1/2}}} \nabla \cdot \mathbf{w}_{h}u_h^{n+1}u_h^{n}~dx \\ &\leq||u_h^n||^2_{L_2(\Omega_{t^{n}})} + \frac{\Delta t}{4}\int_{\Omega_{h, t^{n+1/2}}}\nabla \cdot \mathbf{w}_{h} |u_h^{n+1}+u_h^{n}|^2~dx \\
&    +\frac{\Delta t}{\mu} ||f^{n+1/2}||^2_{{L_2(\Omega_{t^{n+1/2}})}} +\Delta t \sum_{K \in \mathcal{T}_{h,t^{n+1/2}}} \delta_{K} ||f^{n+1/2}||^2_{K}\\
&\leq  \Delta t\int_{\Omega_{h,t^{n+1/2}}} \nabla \cdot \mathbf{w}_{h}\left( \frac{1}{4}|u_h^{n+1}+u_h^{n}|^2 - u_h^{n+1}u_h^{n}\right)~dx \\
&  +  ||u_h^n||^2_{L_2(\Omega_{t^{n}})} +\frac{\Delta t}{\mu} ||f^{n+1/2}||^2_{{L_2(\Omega_{t^{n+1/2}})}} +\Delta t \sum_{K \in \mathcal{T}_{h,t^{n+1/2}}} \delta_{K} ||f^{n+1/2}||^2_{K}\\
&\leq  \frac{\Delta t}{2}\int_{\Omega_{h,t^{n+1/2}}} \nabla \cdot \mathbf{w}_{h}\left(|u_h^{n}|^2 + |u_h^{n+1}|^2 \right)~dx \\
&  +  ||u_h^n||^2_{L_2(\Omega_{t^{n}})} +\frac{\Delta t}{\mu} ||f^{n+1/2}||^2_{{L_2(\Omega_{t^{n+1/2}})}} +\Delta t \sum_{K \in \mathcal{T}_{h,t^{n+1/2}}} \delta_{K} ||f^{n+1/2}||^2_{K}\\
\end{aligned}
\end{equation*}
Using the ALE map and its Jacobian, we obtain
\begin{equation*}
\begin{aligned}
||u_h^{n+1}||^2_{L_2(\Omega_{t^{n+1}})} +& \frac{\Delta t}{8} |||u_h^{n+1} + u_h^n|||^2_{{L_2(\Omega_{t^{n+1/2}})}}  \\
 &\leq   \frac{\Delta t}{2}~ ||\nabla \cdot \mathbf{w}_{h}||_{L_{\infty}(\Omega_{t_{n+1/2}})}~ ||J_{\mathcal A_{{t_{n+1}},~t_{n+1/2}}}||_{L_{\infty}(\Omega_{t_{n+1}})}~||u_h^{n+1}||^2_{L_2(\Omega_{t^{n+1}})} +\frac{\Delta t}{\mu} ||f^{n+1/2}||^2_{{L_2(\Omega_{t^{n+1/2}})}}\\
 & +\Delta t \sum_{K \in \mathcal{T}_{h,t^{n+1/2}}} \delta_{K} ||f^{n+1/2}||^2_{K}+ \frac{\Delta t}{2}~ ||\nabla \cdot \mathbf{w}_{h}||_{L_{\infty}(\Omega_{t_{n+1/2}})}~ ||J_{\mathcal A_{{t_{n}},~t_{n+1/2}}}||_{L_{\infty}(\Omega_{t_{n}})}~ ||u_h^n||^2_{L_2(\Omega_{t^{n}})}
\end{aligned}
\end{equation*}
Denoting
\[
 \beta_1^n   =  \frac{1}{2}~||\nabla \cdot \mathbf{w}_{h}||_{L_{\infty}(\Omega_{t_{n+1/2}})}||J_{\mathcal A_{{t_{n+1}},~t_{n+1/2}}}||_{L_{\infty}(\Omega_{t_{n+1}})}, \qquad
\beta_2^n = \frac{1}{2}~||\nabla \cdot \mathbf{w}_{h}||_{L_{\infty}(\Omega_{t_{n+1/2}})} ||J_{\mathcal A_{{t_{n}},~t_{n+1/2}}}||_{L_{\infty}(\Omega_{t_{n}})},
\]
the inequality becomes
\begin{equation*}
\begin{aligned}
||u_h^{n+1}||^2_{L_2(\Omega_{t^{n+1}})} + \frac{\Delta t}{8} |||(u_h^{n+1} + u_h^n)|||^2_{{L_2(\Omega_{t^{n+1/2}})}}  
&\leq  \Delta t \beta_1^{n+1}||u_h^{n+1}||^2_{{L_2(\Omega_{t^{n+1}})}} + ( 1 + \Delta t \beta_2^n)||u_h^n||^2_{L_2(\Omega_{t^{n}})}\\
&  +\frac{\Delta t}{\mu} ||f^{n+1/2}||^2_{{L_2(\Omega_{t^{n+1/2}})}} +\Delta t \sum_{K \in \mathcal{T}_{h,t^{n+1/2}}} \delta_{K} ||f^{n+1/2}||^2_{K}
\end{aligned}
\end{equation*}
Summing over the index $n = 0, 1,2,...N$, and using the assumption on $\delta_k$ , we have
\begin{equation*}
\begin{aligned}
||u_h^{N+1}||^2_{L_2(\Omega_{t^{N+1}})} +& \frac{\Delta t}{8} \sum_{n=0}^{N} |||(u_h^{n+1} + u_h^n)|||^2_{{L_2(\Omega_{t^{n+1/2}})}}  \\
&\leq  \Delta t \beta_1^{N+1}||u_h^{N+1}||^2_{{L_2(\Omega_{t^{N+1}})}} + \Delta t \sum_{n=1}^{N} (\beta_1^n + \beta_2^n)||u_h^n||^2_{{L_2(\Omega_{t^{n}})}} + ( 1 + \Delta t \beta_2^0)||u_h^0||^2_{L_2(\Omega_{t^{0}})}\\
 &  +\frac{\Delta t}{\mu} ||f^{n+1/2}||^2_{{L_2(\Omega_{t^{n+1/2}})}} +\Delta t \sum_{K \in \mathcal{T}_{h,t^{n+1/2}}} \delta_{K} ||f^{n+1/2}||^2_{K}\\
 & \leq \Delta t \sum_{n=1}^{N+1} (\beta_1^n + \beta_2^n)||u_h^n||^2_{{L_2(\Omega_{t^{n}})}} + ( 1 + \Delta t \beta_2^0)||u_h^0||^2_{L_2(\Omega_{t^{0}})} +\Delta t\sum_{n=1}^{N+1}\left(\frac{2}{\mu} + \frac{\Delta t}{2} \right) \|f^{n+1/2}\|^2_{L^{2}(\Omega_{t_{n+1/2}})}
\end{aligned}
\end{equation*}
Finally, using the Grownwall's lemma, we get
\begin{equation*}
\begin{aligned}
  \|u_h^{N+1}\|&^2_{L^2(\Omega_{N+1})}  + \frac{\Delta t}{4} \sum_{n=0}^{N}|||(u_h^{n+1} + u_h^n)|||^{2}_{L^2 \left(\Omega_{t^{n+1/2}} \right)} \\
&\leq  \left((1+\Delta t \beta_2^0)\| u_h^{0}\|^2_{L^2(\Omega_{0})}  + \Delta t\sum_{n=0}^{N}\left(\frac{2}{\mu}+  \Delta t \right) \|f^{n+1/2}\|^2_{L^{2}(\Omega_{t_{n+1/2}})}\right)
 \exp{\left( \Delta   t \sum_{n=1}^{N+1}\frac{\beta_1^n + \beta_2^n}{1-\Delta t(\beta_1^n + \beta_2^n)} \right)}. \\
 \end{aligned}
\end{equation*}
  with a restriction on $\Delta t$ as,
\[
 \Delta t < \frac{1}{\beta^n_1 +\beta_2^n} = \left(||\nabla \cdot \mathbf{w}_{h}||_{L_{\infty}(\Omega_{t_{n+1/2}})}||J_{\mathcal A_{{t_{n+1/2}},~t_{n+1}}}||_{L_{\infty}(\Omega_{t_{n+1}})} + ||\nabla \cdot \mathbf{w}_{h}||_{L_{\infty}(\Omega_{t_{n+1/2}})} ||J_{\mathcal A_{{t_{n}},~t_{n+1/2}}}||_{L_{\infty}(\Omega_{t_{n}})} \right)^{-1}.
\]
\end{proof}

\section{Numerical results}
Numerical results for the proposed conservative ALE-SUPG finite element are presented in this section. Two examples, $(i)$ transient scalar equation   with  $\epsilon=0.01$, $\bb=0$ and $c=0$ in \eqref{model}, and $(ii)$ transient scalar equation $\epsilon=10^{-8}$, $\bb=(1,0)^T$ and $c=0$ in \eqref{model}, are considered. The standard Galerkin solution and the SUPG solution  are compared. In computations, the SUPG parameter is chosen as
\begin{align*}
 \delta_K &= \ds\left\{ \begin{array}{ccl}
       \ds\frac{\delta_0{h_{K,t}}}{ {\|\bb -\bw\|_{L^{\infty}}} } & \;\text{  if}& \epsilon < h_{K,t}\|\bb -\bw\|_{L^{\infty}},\\
           0 &\;\text{  else,}&
        \end{array}\right. 
\end{align*}
where $\delta_0$ a numerical parameter and $h_{K,t}$ is the time-dependent local cell size. Computations are performed for different values of $\delta_0$. Further, the overshoots and undershoots are plotted.  All computations are performed using an unstructured triangular mesh. Further, the piecewise linear and piecewise quadratic finite elements are used for the spatial discretization in the first and second examples, respectively.  Even though the second derivative in the SUPG formulation becomes zero for the linear finite elements, the influence will be negligible for a very small diffusive coefficient $\epsilon$. Note that the SUPG method is needed only for problems with small diffusion coefficient.

\subsection{Example 1}
We consider the time-dependent equation  \eqref{model} with  $\epsilon=0.01$, $\bb=0$  and $c=0$. Further, the initial value is chosen as,  $u_0=1600~ Y_1(1-Y_1)~Y_2(1-Y_2)$ and  $\Omega_0:=(0,1)^2$  is the initial (reference) domain. Moreover, 
 the deformation of the time-dependent domain, $\Omega_t$ is defined by  
\[ x(Y, t) =  \mathcal{A}_t(Y) : \left\{
  \begin{array}{l l}
    x_1 = Y_1(2 - cos(20 \pi t)) & \\
    x_2 = Y_2(2 - cos(20 \pi t))
  \end{array} \right.    
\]
 where $Y\in\Omega_0$. Then, the mesh velocity $\mathbf{w}$ becomes
\[ 
  \mathbf{w} = \frac{dY}{dt} = \left( \frac{20 \pi x_1 sin(20 \pi t)}{2-cos(20 \pi t)}, \frac{20 \pi x_2 sin(20 \pi t)}{2-cos(20 \pi t)}  \right).
\]
%
In computations, we use the piecewise linear in time interpolation for the domain movement, i.e., for every $\tau\in[t^n,t^{n+1}]$
define $x_{h}(Y,t) $ by
\[
 x_{h}(Y,\tau) = \frac{\tau-t^n}{\Delta t}x_{h}^{n+1}(Y) +  \frac{t^{n+1}-\tau }{\Delta t}x_{h}^{n}(Y).
\]
Hence, the mesh velocity is obtained as
\[
 \bw_{h}(Y,\tau) =  \frac{x_{h}^{n+1}(Y) - x_{h}^{n}(Y)}{\Delta t}.
\]
\begin{figure}[ht!]
\begin{center}
\unitlength1cm
\begin{picture}(11.5, 5.)
\put(-1.,-0.5){\makebox(6,6){\includegraphics[scale=0.25]{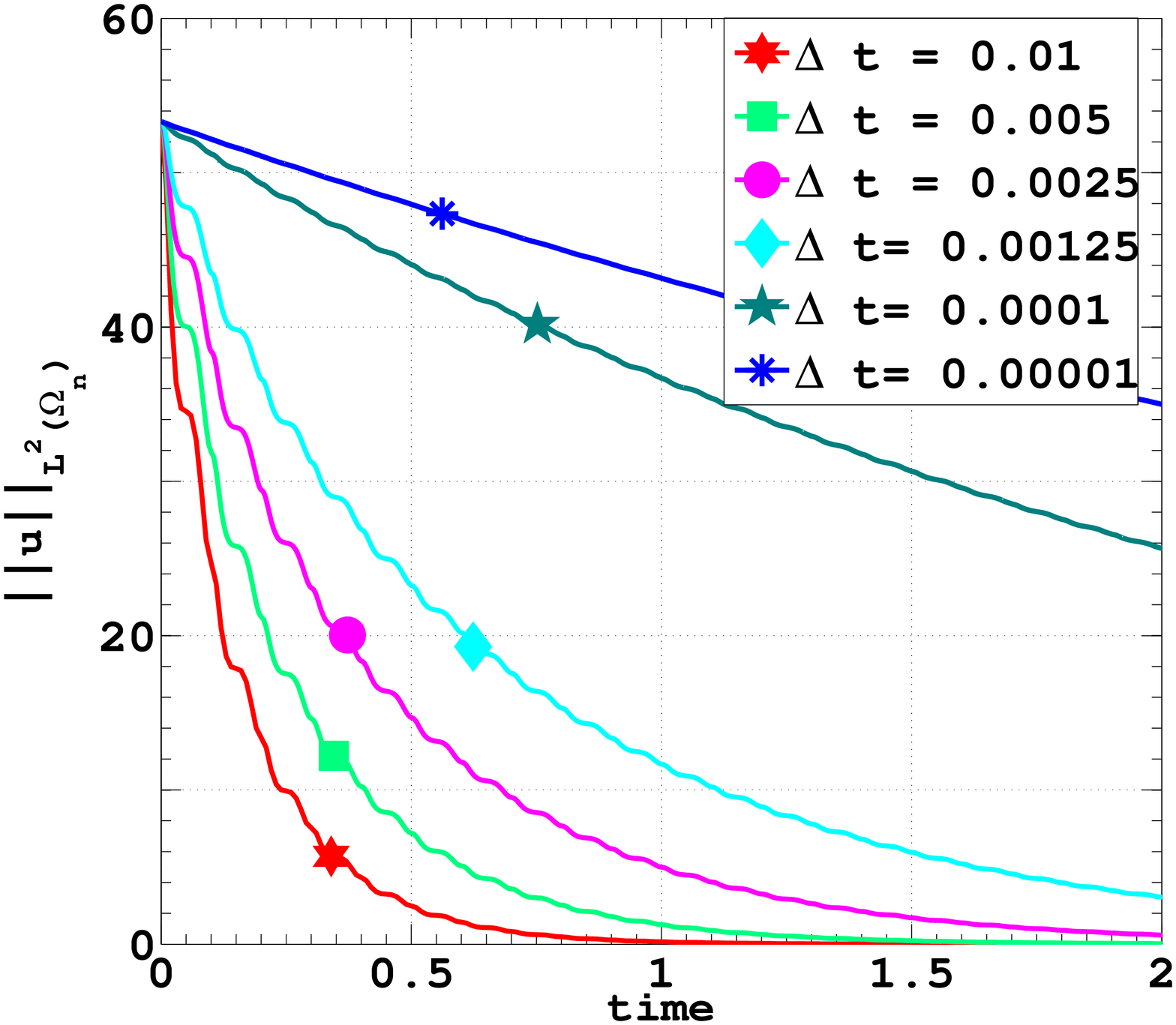}}}
\put(6,-0.5){\makebox(6,6){\includegraphics[scale=0.25]{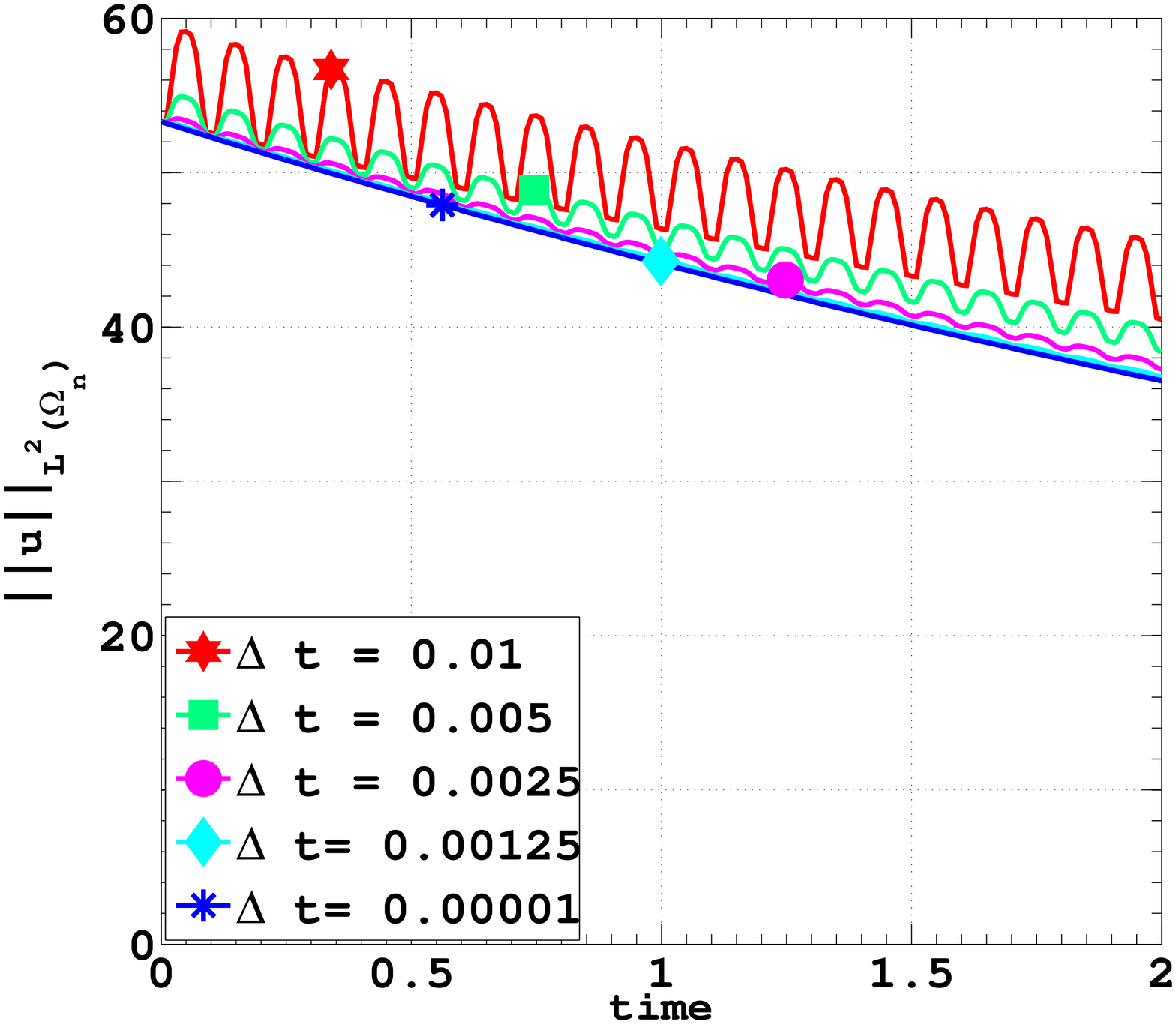}}}
\put(2,5.5){(a)}
\put(9,5.5){(b)}
\end{picture}
\end{center}
\caption{$L^2$ norm of the solution obtained with the standard Galerkin solution for different time-steps.  Implicit Euler (a), and Crank-Nicolson (b). \label{Ex1Gal}}
\end{figure}
\begin{figure}[ht!]
\begin{center}
\unitlength1cm
\begin{picture}(11.5,6.)
\put(-1.,-0.5){\makebox(6,6){\includegraphics[scale=0.25]{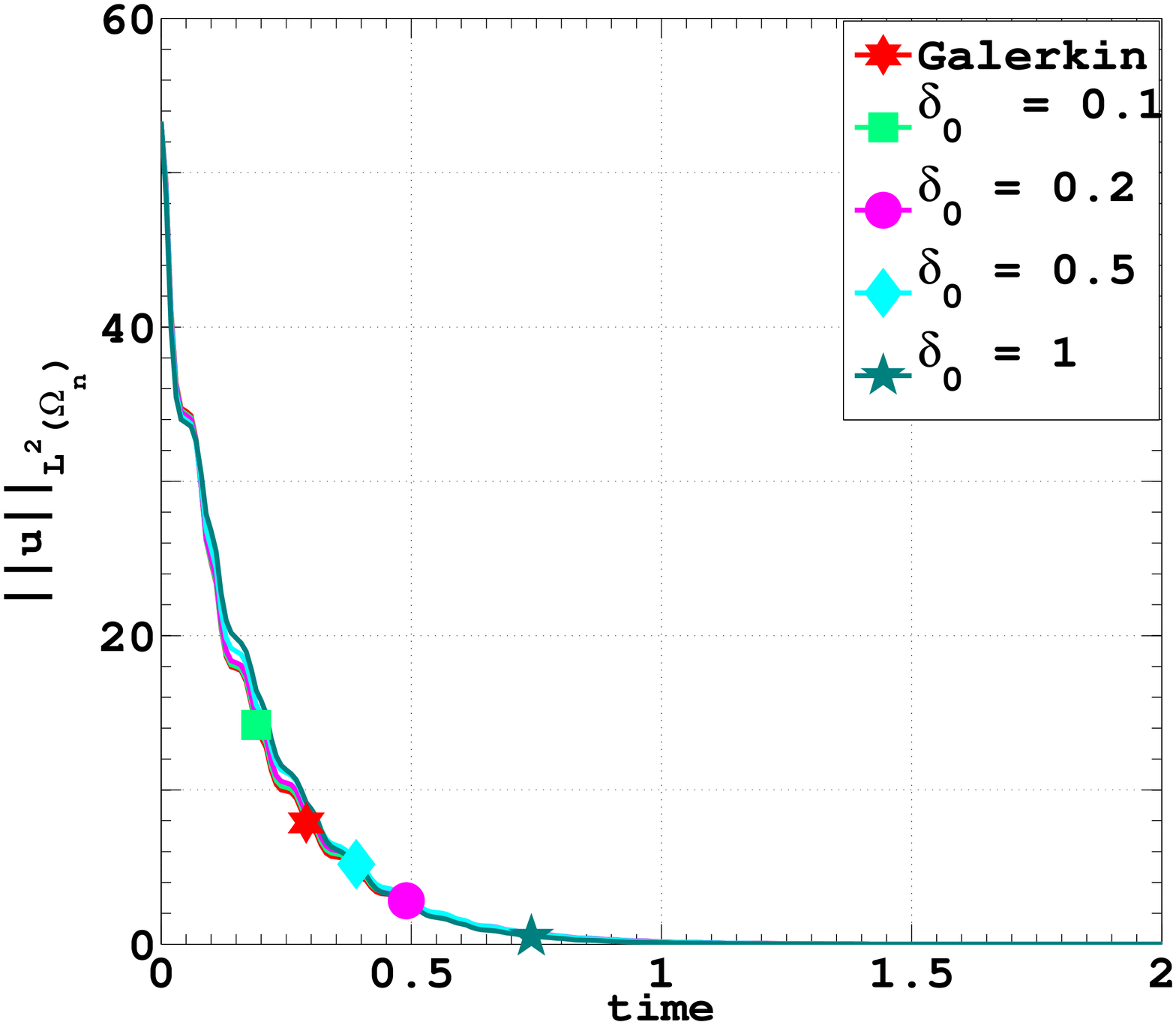}}}
\put(6,-0.5){\makebox(6,6){\includegraphics[scale=0.25]{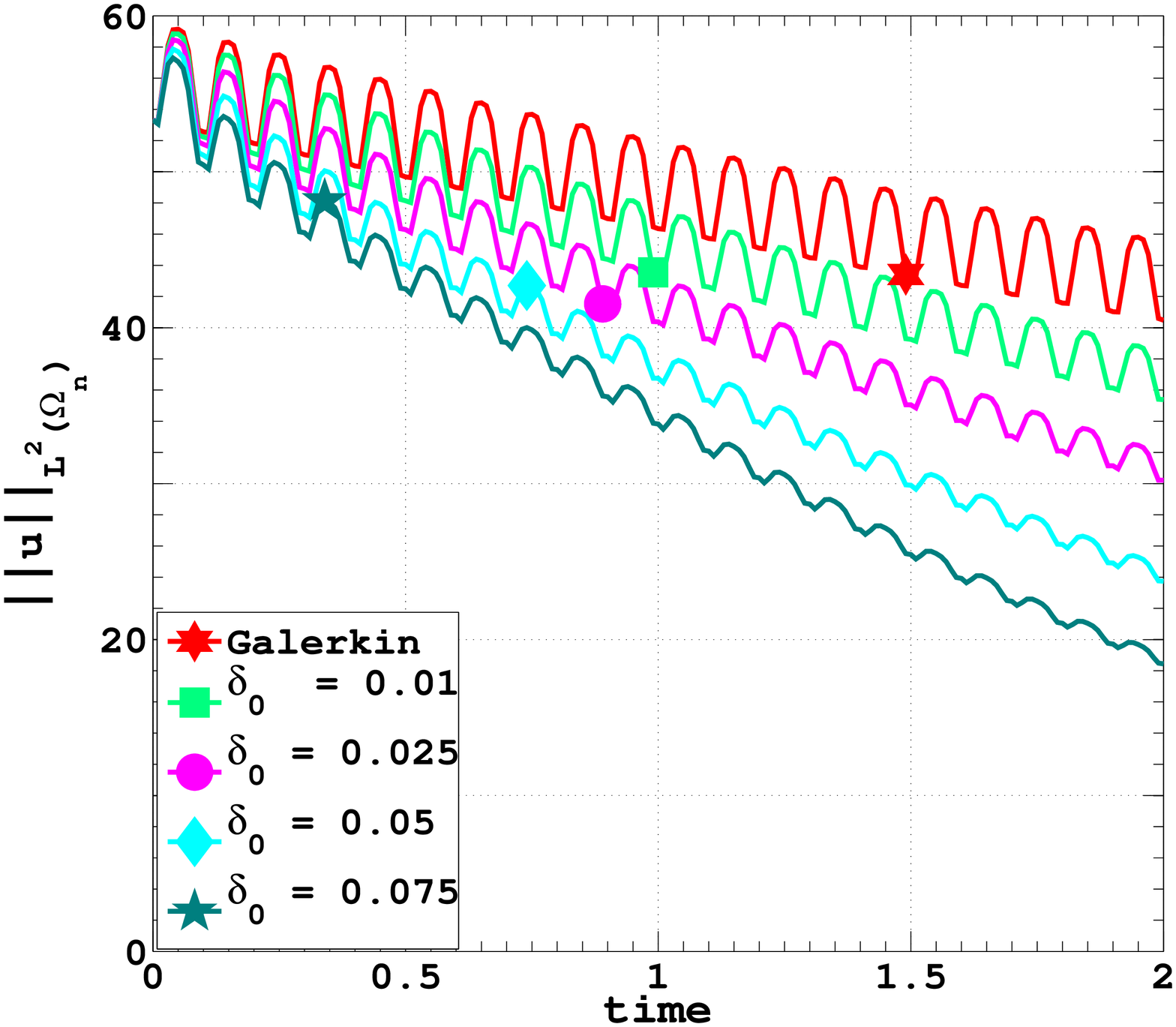}}}
\put(2,5.5){(a)}
\put(9,5.5){(b)}
\end{picture}
\end{center}
\caption{$L^2$-norm of the solution obtained with the SUPG discretization for different $\delta_0$. Implicit Euler (a),  and Crank-Nicolson (b). \label{Ex1SUPG}}
\end{figure}
The computational mesh consists 8192 triangles and 4225 degrees of freedom (DOF). Even though the convective term is zero in the considered example, a convection type term will be introduced by the mesh velocity due to the conservative ALE formulation.

The  $L^2-$norm of the  solution obtained with  the standard Galerkin for different time-steps are presented in Figure~\ref{Ex1Gal}. The numerical solution obtained with the Euler method is more diffusive, and it decreases monotonically as it can be clearly seen from the stability estimates lemma~\eqref{stabcons}. However, the diffusive effect is not observed when a smaller time-step is used, see Figure~\ref{Ex1Gal} (a). Though the solution obtained with a large time-step is oscillatory in the case of Crank-Nicolson time discretization, the solution is not as diffusive as in the Euler's method. Nevertheless, the influence of time-steps on the solution is less when the Crank-Nicolson method is used. 

Next, $L^2-$norm of the solution obtained with the SUPG discretization  for different $\delta_0$ are presented in  Figure~\ref{Ex1SUPG}. Since $\bb=0$,  the only term in convection is the mesh velocity. Therefore, the SUPG parameter is calculated using
\begin{align*}
 \delta_K &= \ds\left\{ \begin{array}{ccl}
       \ds\frac{\delta_0{h_{K,t}}}{ {\|\bw\|_{L^{\infty}}} } & \;\text{  if}& \epsilon < h_{K,t}\|\bw\|_{L^{\infty}},\\
           0 &\;\text{  else.}&
        \end{array}\right. 
\end{align*}
Further, the time-step $\Delta t=0.01$ is used. Since the solution obtained with the Euler method is already too diffusive, the smearing effect in the SUPG solution is not visible explicitly. However, the effects of $\delta_0$ can be seen clearly in the solution obtained with the Crank-Nicolson method, see Figure~\ref{Ex1SUPG}~(b). Further, the amplitude of the oscillation in the $L^2$-norm of the solution reduced when $\delta_0$ increased.



\subsection{Example 2}
We next  consider an example that exemplifies a fluid-structure interaction problem. 
Let 
\[
 \Omega_0^S := \left\{(Y_1,Y_2)\in\mathbb{R}^2;~Y_1^2+Y_2^2\le1 \right\}\quad \text{ and } \quad 
 \Omega_t^S := \left\{(x_1,x_2) \right\}\subset\mathbb{R}^2,
\]
be the reference and the time-dependent circular disc, respectively. Here, the time-dependent coordinates $(x_1,x_2)$ are defined by 
\[ x(Y, t) =  \mathcal{A}_t(Y) : \left\{
  \begin{array}{l l}
    x_1 = Y_1  & \\
    x_2 = Y_2 +0.5\sin(2\pi t/5).
  \end{array} \right.    
\]
We then define a time-dependent two-dimensional channel 
\[
 \Omega_t:=\{(-3,9)\times(-3,3)\}\setminus \bar{\Omega}_t^S
\]
that excludes a periodically oscillating (up and down) circular disc $\Omega_t^S$.
Further,  we define  $\Gamma_N:=
\{9\}\times(-3,3)$ as the   out flow boundary and $\Gamma_D:= \partial\Omega_t \setminus \Gamma_N$ as the Dirichlet boundary.
\begin{figure}[ht!]
\begin{center}
\unitlength1cm
\begin{picture}(4,5.5)
\put(0.0,0.25){\makebox(4,4){\includegraphics[scale=0.4]{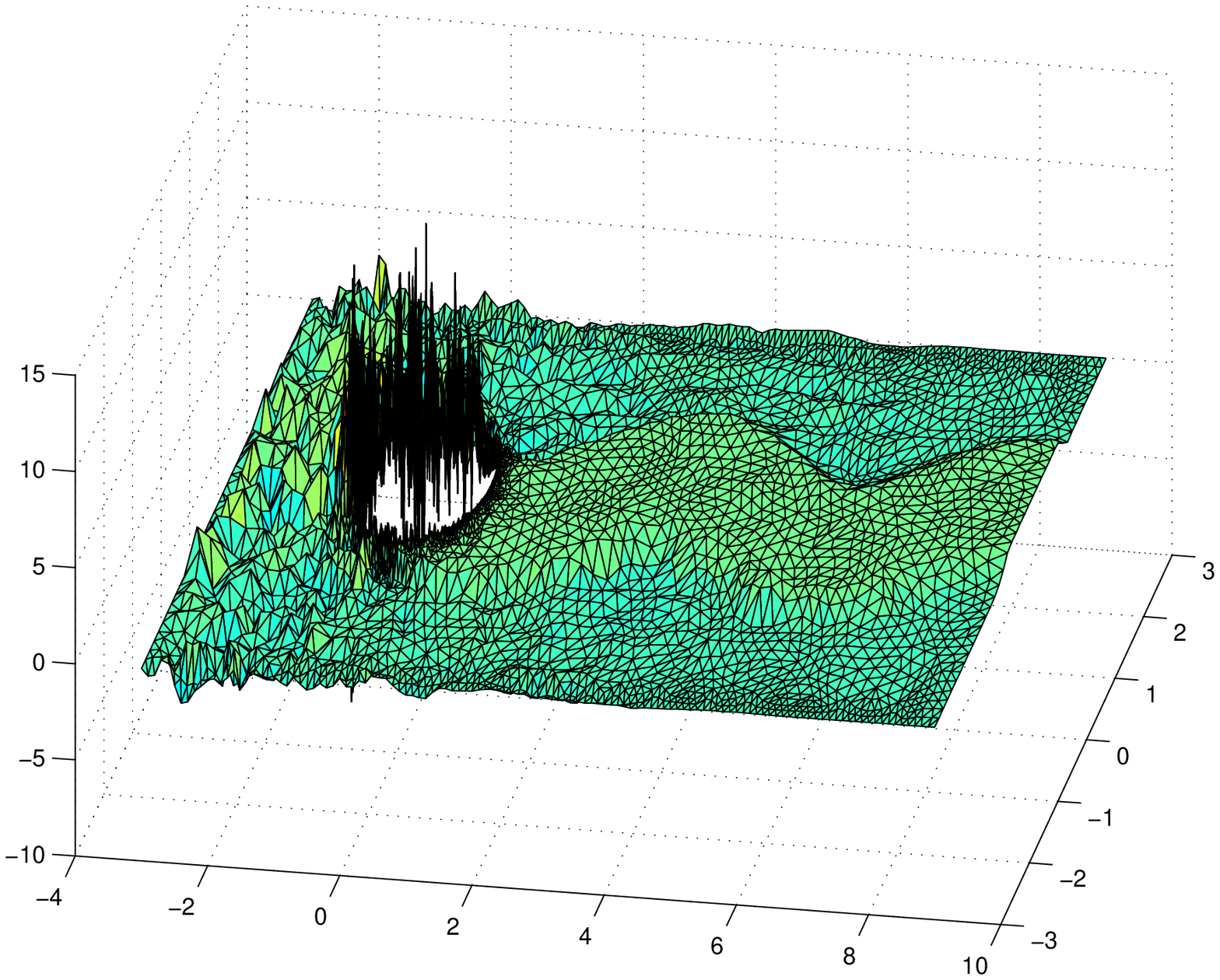}}}
\put(1.5,2.3){\makebox(4,4){\includegraphics[scale=0.2]{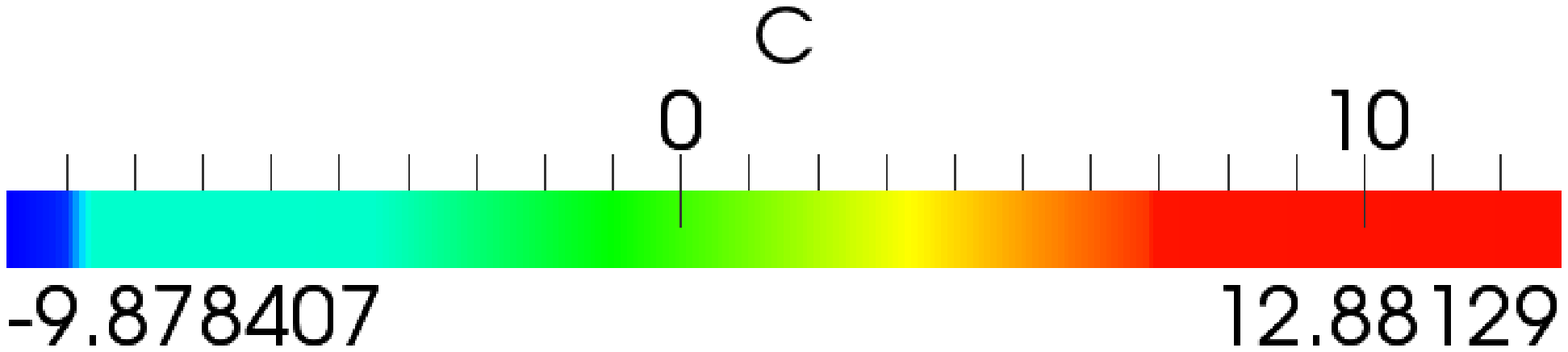}}}
\end{picture}
\end{center}
\caption{Standard Galerkin solution for the Example~2 at $t=10$. The overshoots and undershoots are above $100 \%$. \label{Ex2Gal}}
\end{figure}
We now solve  the transient scalar equation \eqref{model} with $\epsilon=10^{-8}$, $\bb=(1,0)^T$ and $c=0$.
Further,
we impose the homogeneous Neumann condition on $\Gamma_N$, and      
\[
u_D(x_1,x_2) =
	\begin{cases}
	1 \qquad \text{on}~ \partial \Omega_t^S, \\
	0 \qquad   \text{on } \Gamma_D.
	\end{cases}
\]
Note that there will be a boundary layer on the inlet side of the oscillating circular disc, and two interior layers behind the disc. Since the solid disc oscillates periodically, the position of the boundary and the interior layers also change in time.

The  computations are performed until the dimensionless time $T=10$ with the time step $\Delta t=0.01$. Further, the linear elastic-solid update technique is used to handle the mesh movement that occurs due to the oscillations of the solid disc.  At each time step, we first compute the displacement of the disc. We then solve the linear elastic equation in $\Omega_{t^n}$ to compute the inner points' displacement by considering the displacement on $\partial \Omega_{t^{n+1}}^S$ as the Dirichlet value. This elastic update technique avoids the remeshing during the entire simulation. 
The considered triangulated domain for this example consists 9416 triangular cells and 19552 DOF.
As expected the solution obtained with the standard Galerkin discretization consists spurious oscillations and instabilities, see Figure~\ref{Ex2Gal}. 

 We next perform an array of computations with different values of  $\delta_0$. Since the solution for this example, $u\in[0,~1]$, the values of the numerical solution below $0$ and above $1$ are called undershoots and overshoots, respectively. The observed undershoots and the overshoots for different values of $\delta_0$  are plotted in Figure~\ref{Ex2Under} and~\ref{Ex2Over}, respectively. The oscillations in the overshoots obtained with the Crank-Nicolson time discretization using $\delta_0 = 10 $~and~ $50$ are more, and therefore, a curve fitting is used to plot the overshoots.
 \begin{figure}[ht!]
\begin{center}
\unitlength1cm
\begin{picture}(11.5,6.)
\put(-1.,-0.5){\makebox(6,6){\includegraphics[scale=0.25]{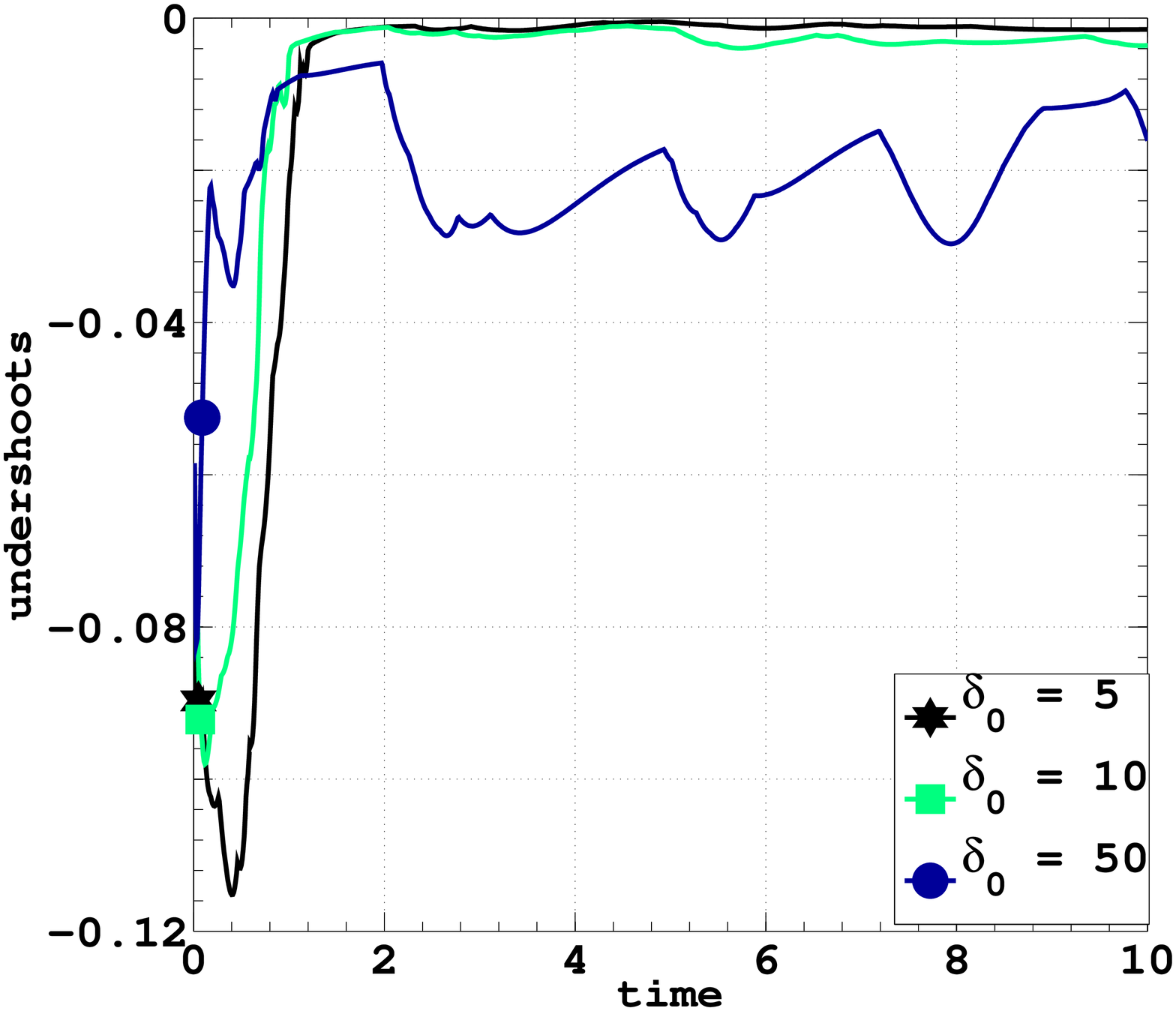}}}
\put(6,-0.5){\makebox(6,6){\includegraphics[scale=0.25]{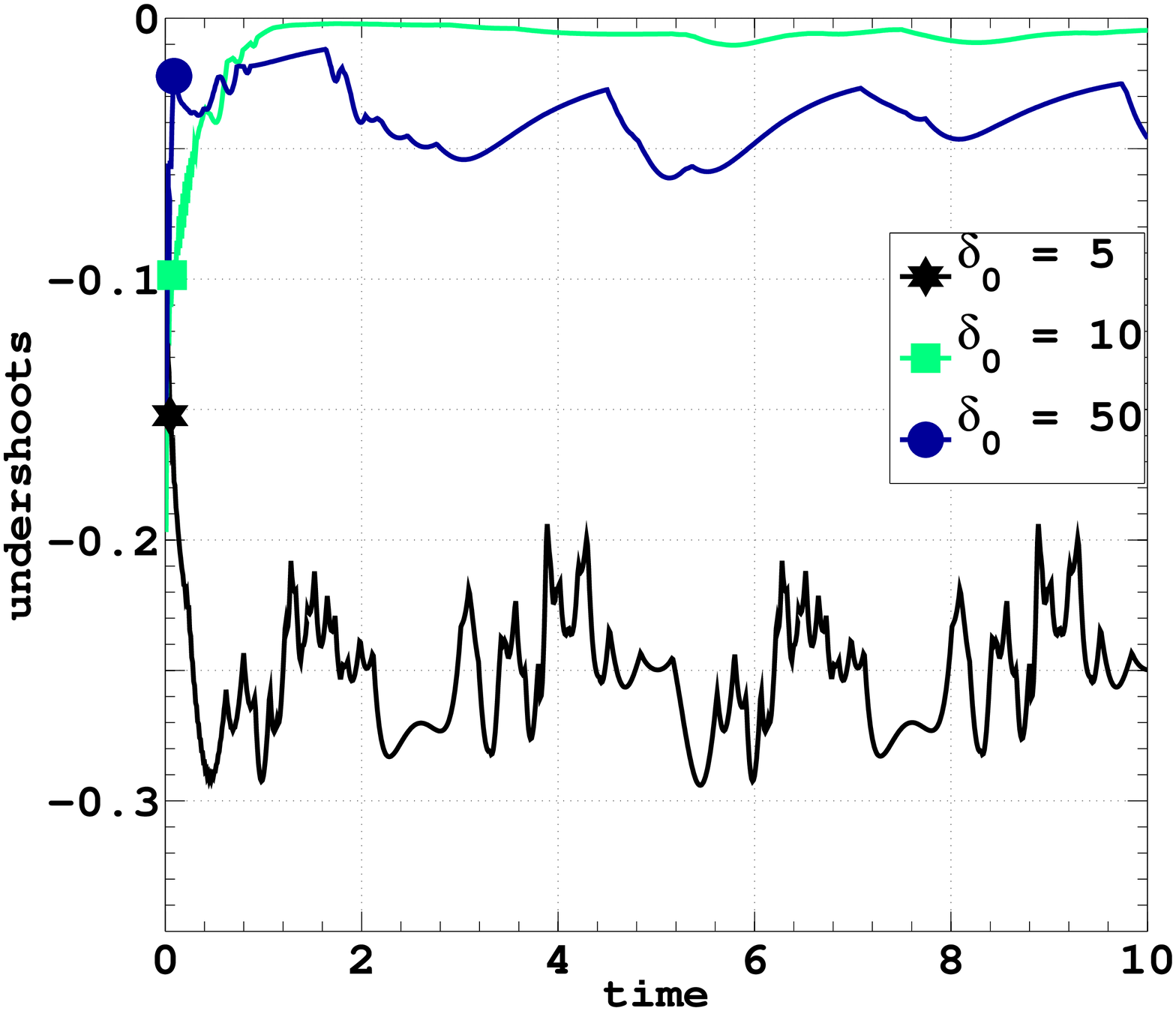}}}
\put(2,5.5){(a)}
\put(9,5.5){(b)}
\end{picture}
\end{center}
\caption{The observed undershoots in the SUPG solution of Example~2 for different values of $\delta_0$. Implicit Euler (a), and Crank-Nicolson (b). \label{Ex2Under}}
\end{figure}
\begin{figure}[ht!]
\begin{center}
\unitlength1cm
\begin{picture}(11.5,6.)
\put(-1.,-0.5){\makebox(6,6){\includegraphics[scale=0.25]{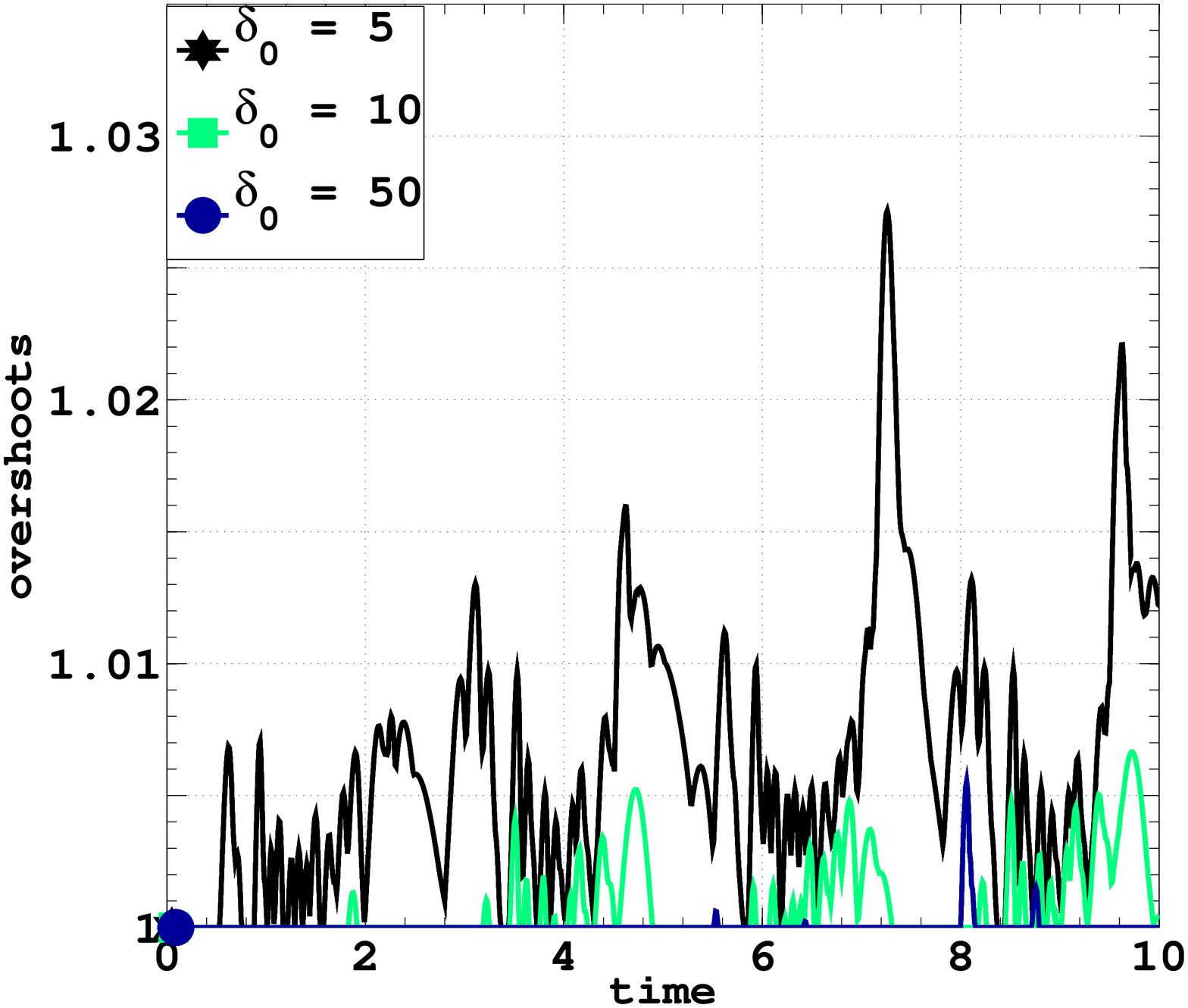}}}
\put(6,-0.5){\makebox(6,6){\includegraphics[scale=0.25]{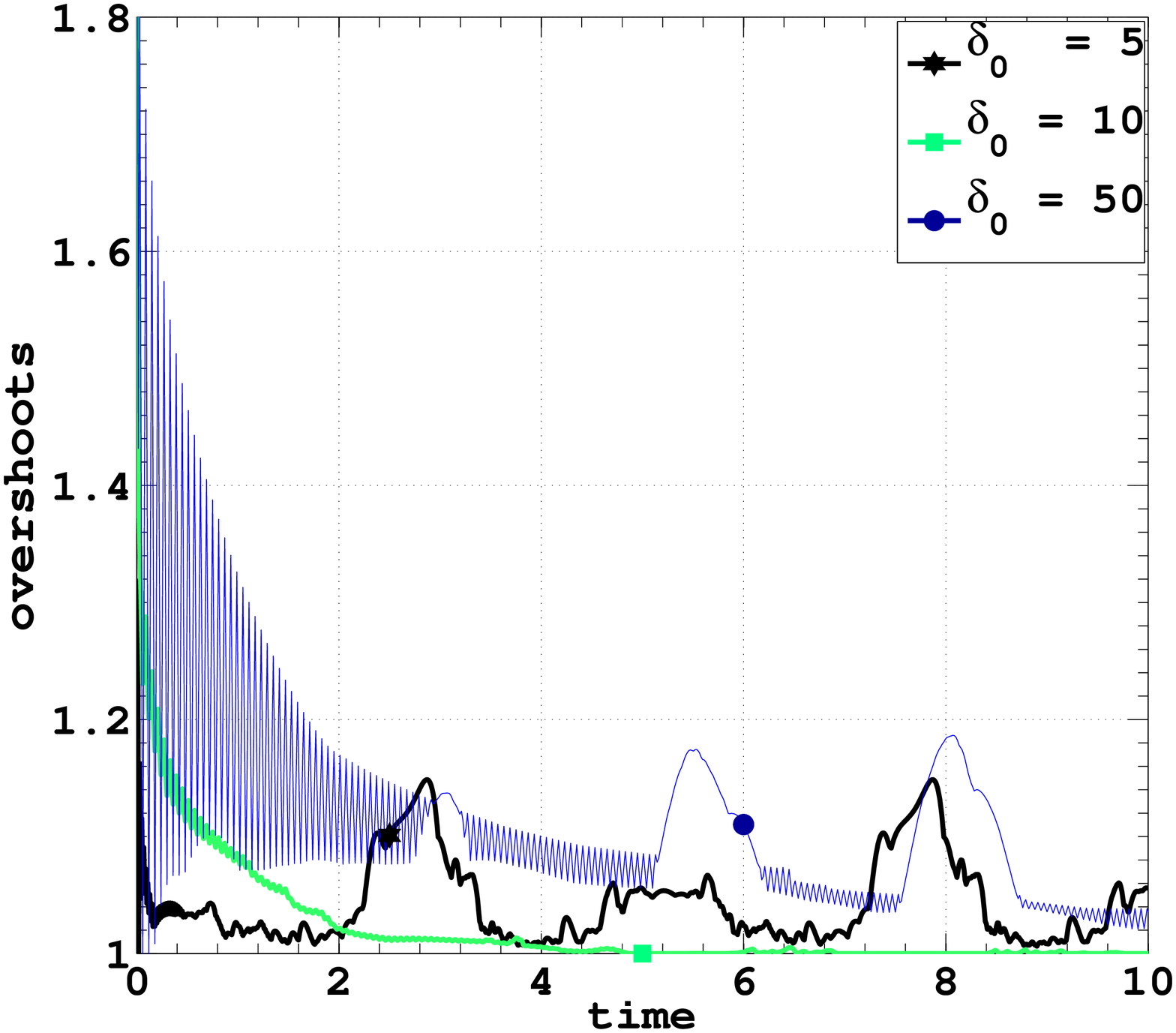}}}
\put(2,5.5){(a)}
\put(9,5.5){(b)}
\end{picture}
\end{center}
\caption{The observed overshoots in the SUPG solution of Example~2 for different values of $\delta_0$. Implicit Euler (a), and Crank-Nicolson (b). \label{Ex2Over}}
\end{figure}
As observed in the previous example, the undershoots and overshoots are less in the Euler's method (note that the scaling of figures are different). For both the Euler and Crank-Nicolson methods, the choice of $\delta_0= 10$ suppresses the undershoots and overshoots more or less. 
Nevertheless, the oscillations can further be suppressed by varying (increasing) $\delta_0$. However, the smearing effect will be more when a large value of  $\delta_0$ is used. Moreover, the plots of the undershoots and overshoots provide only an indication for the choice of $\delta_0$ to suppress the spurious oscillations in the numerical solution.  Here, the smearing effect of $\delta_0$ in the numerical solution is not visible in Figure~\ref{Ex2Under} and~\ref{Ex2Over}. Therefore, to analyze the smearing effect, the obtained ALE-SUPG solution over the line $y=0$ for different values of $\delta_0$ at time $t=10$ are plotted  in Figure~\ref{Ex2y0}.
 \begin{figure}[ht!]
\begin{center}
\unitlength1cm
\begin{picture}(11.5,6.)
\put(-1.,-0.5){\makebox(6,6){\includegraphics[scale=0.25]{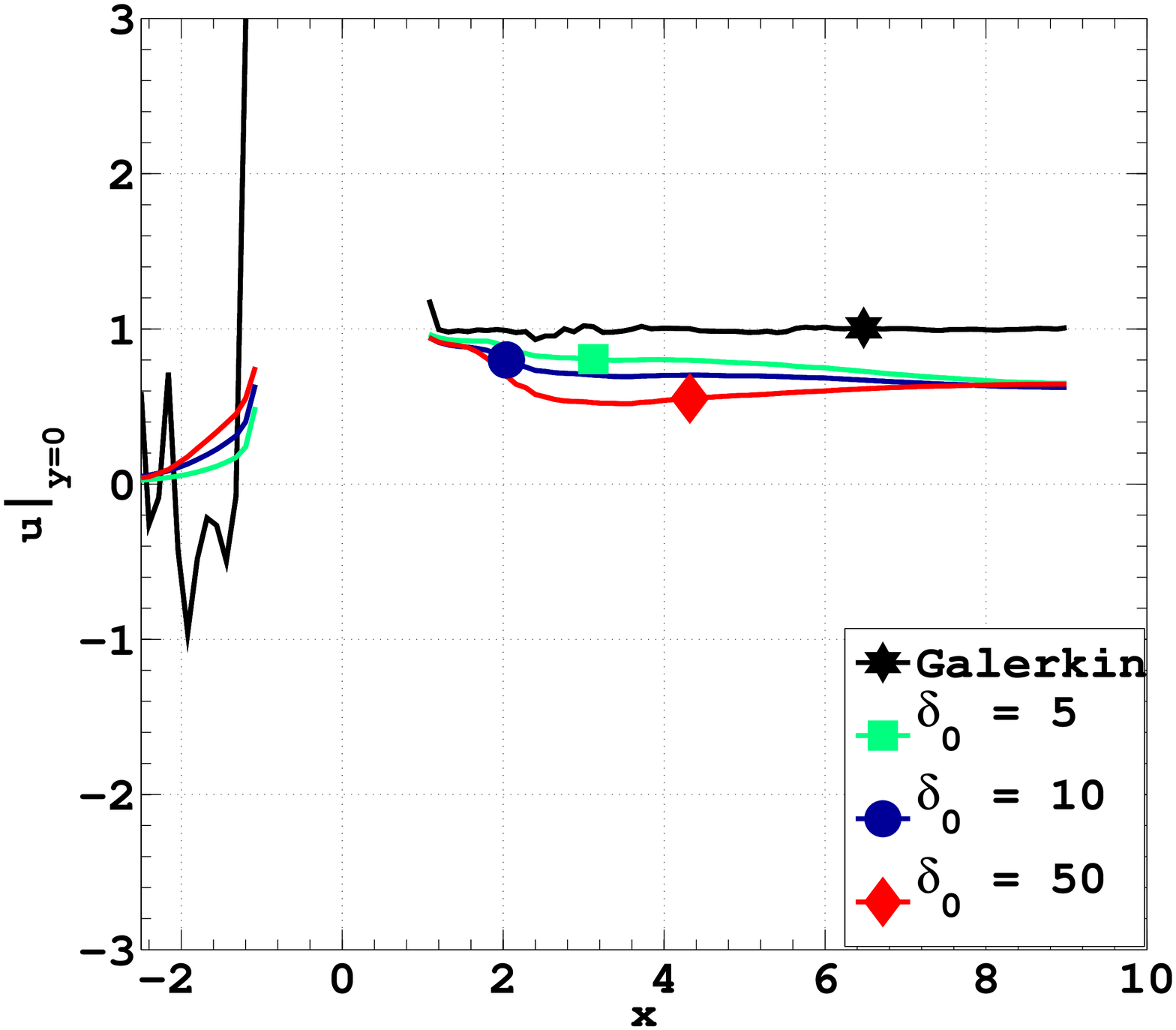}}}
\put(6,-0.5){\makebox(6,6){\includegraphics[scale=0.25]{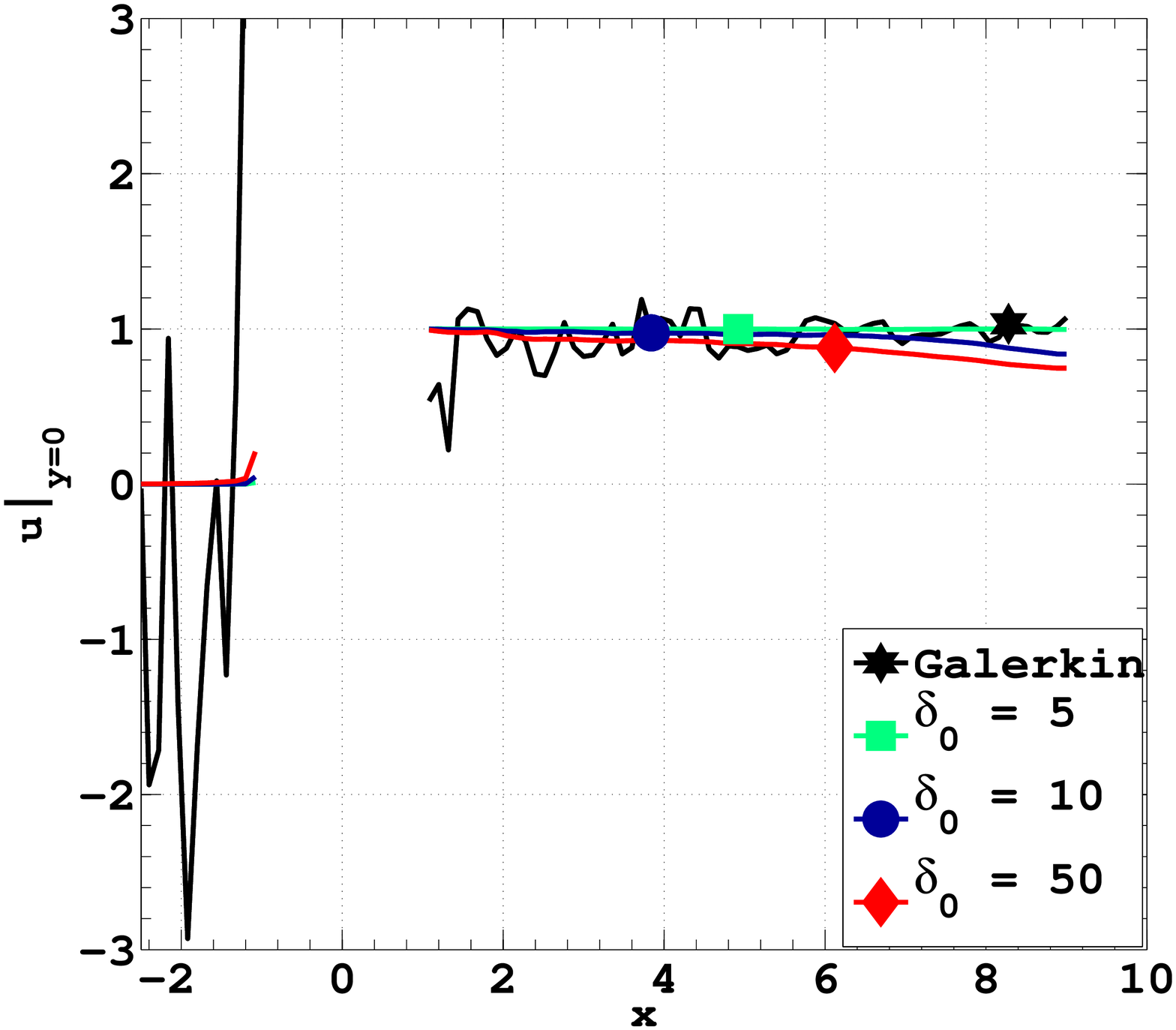}}}
\put(2,5.5){(a)}
\put(9,5.5){(b)}
\end{picture}
\end{center}
\caption{SUPG solution over the line $y=0$ of the Example~2 for different values of $\delta_0$. Implicit Euler (a), and Crank-Nicolson (b). \label{Ex2y0}}
\end{figure}
Based on these observations, we choose $\delta_0=10$ as an optimal value. Next, the  surface plot of the SUPG solution at different instances are plotted in Figure~\ref{Ex2SupgEuler} and  Figure~\ref{Ex2SupgCrank}. Even though, the SUPG approximation suppressed the spurious oscillations in the numerical solution almost, there are very small undershoots and overshoots (approximately $10 \%$) for the chosen $\delta_0=10$.  We could reduce these undershoots and overshoots by increasing   $\delta_0$ further, however, it will smear the solution. This is a well known behavior of the SUPG method in stationary domains. Nevertheless, the oscillations in the solution are suppressed.
\begin{figure}[ht!]
\begin{center}
\unitlength1cm
\begin{picture}(14,11)
\put(0,3.8){\makebox(6,6){\includegraphics[width=7cm]{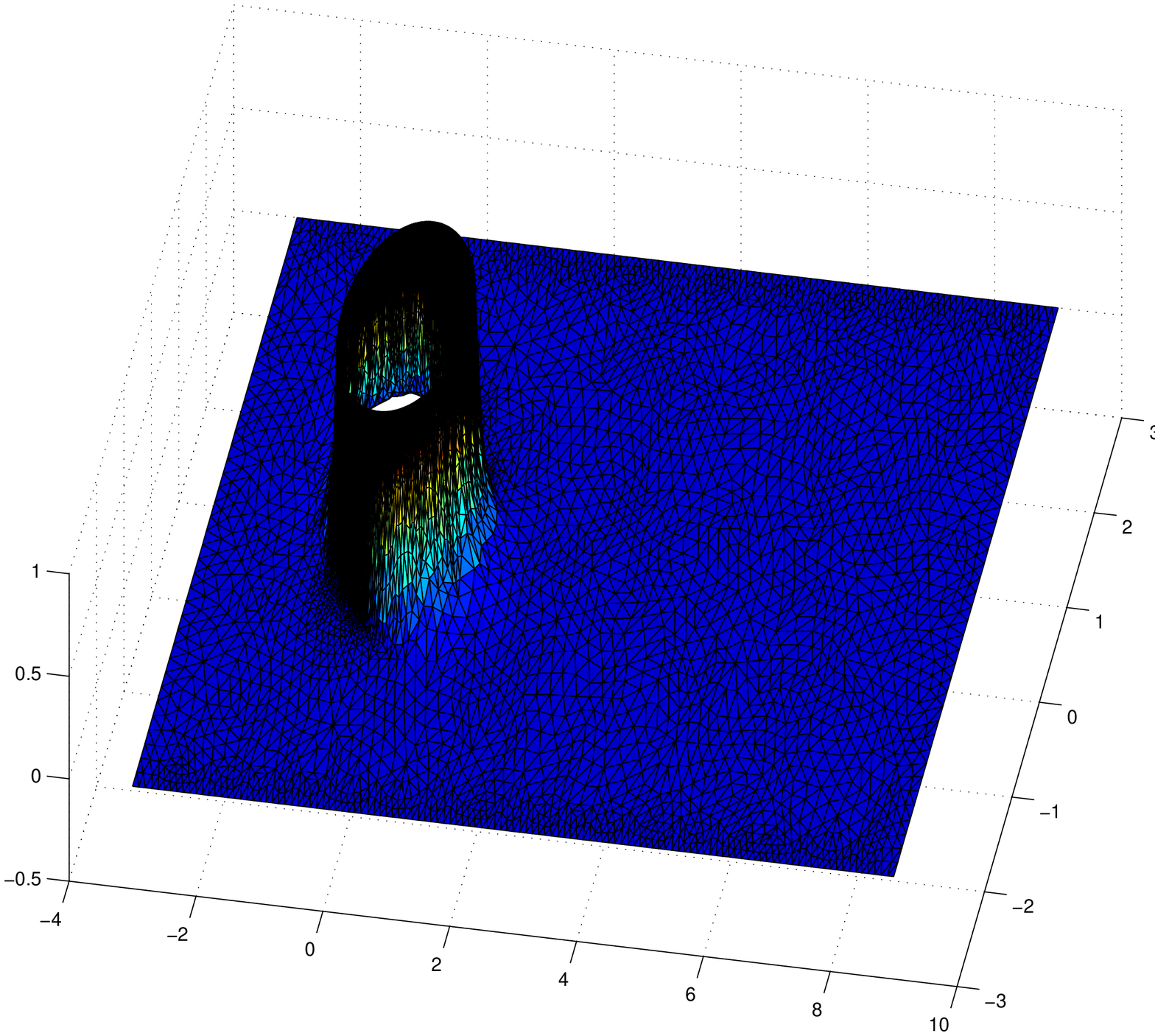}}}
\put(1.21,5.75){\makebox(6,6){\includegraphics[width=3.3cm]{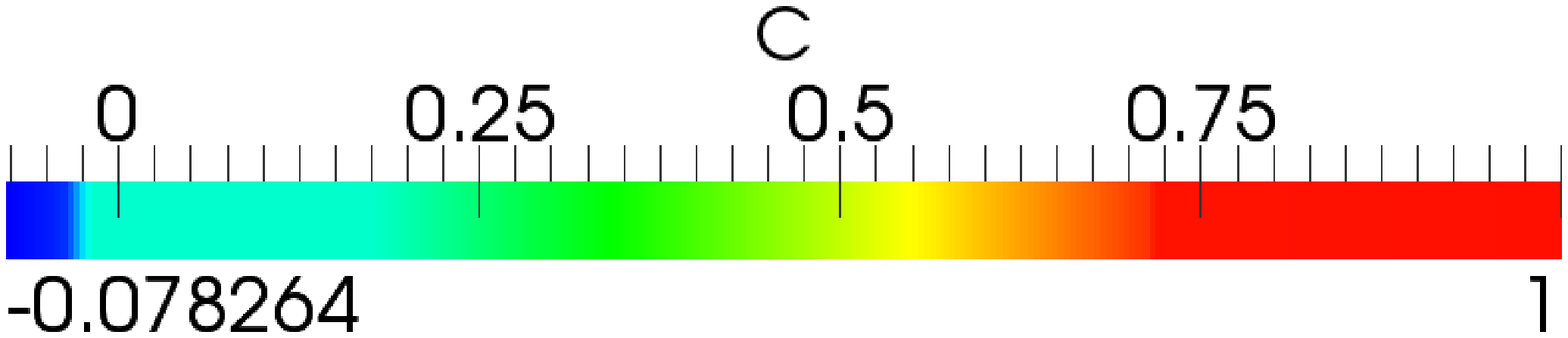}}}
\put(7.5,3.8){\makebox(6,6){\includegraphics[width=7cm]{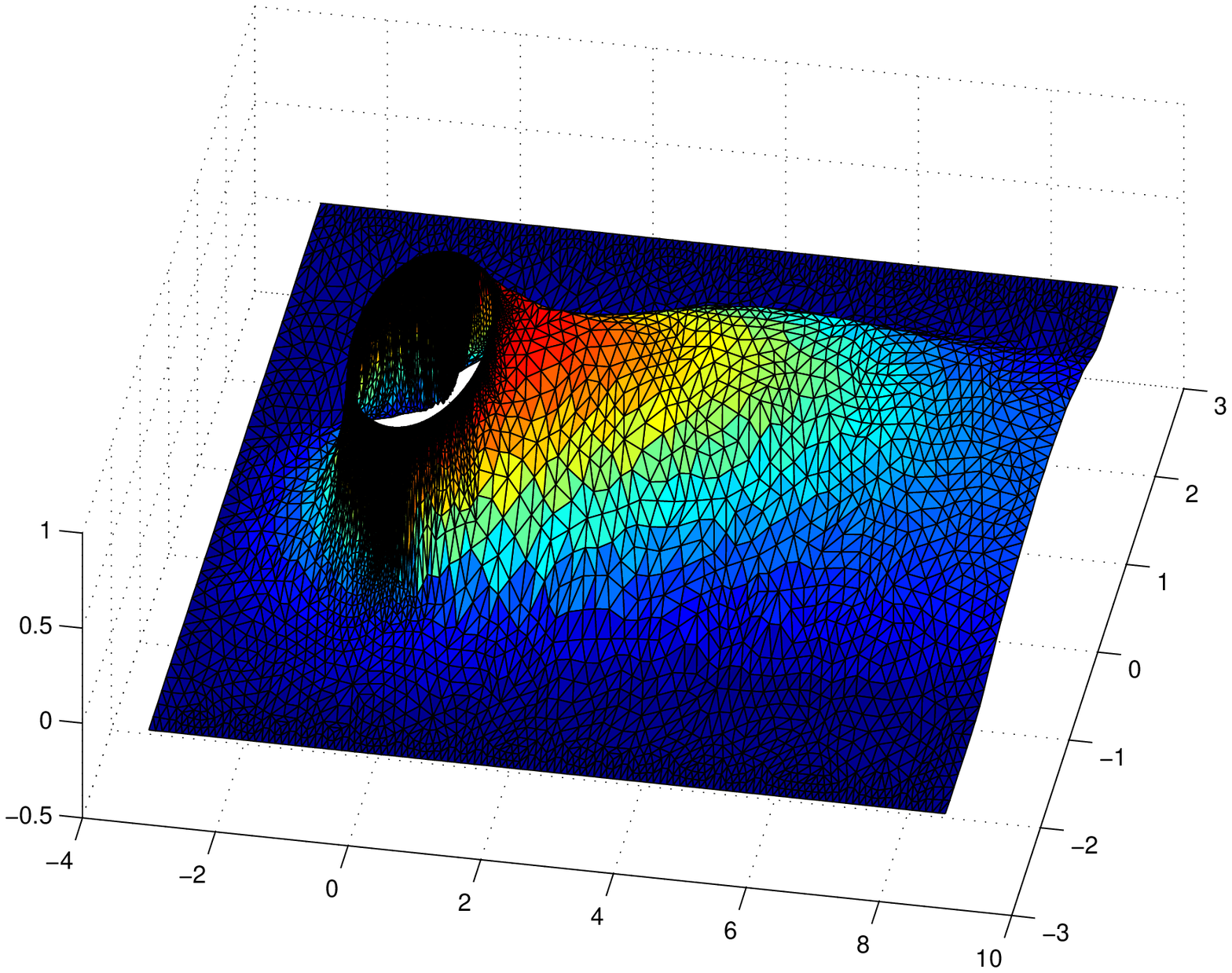}}}
\put(8.71,5.5){\makebox(6,6){\includegraphics[width=3.3cm]{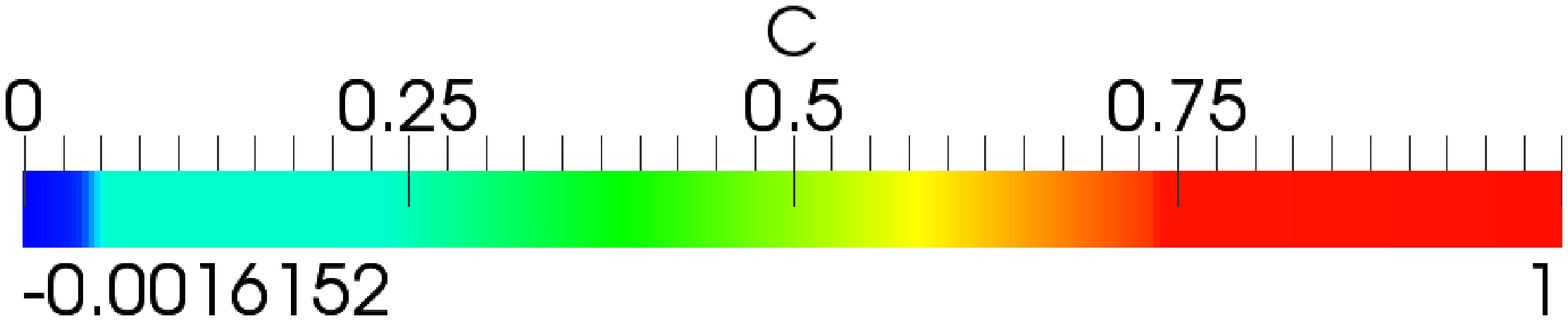}}}
\put(0,-1.5){\makebox(6,6){\includegraphics[width=7cm]{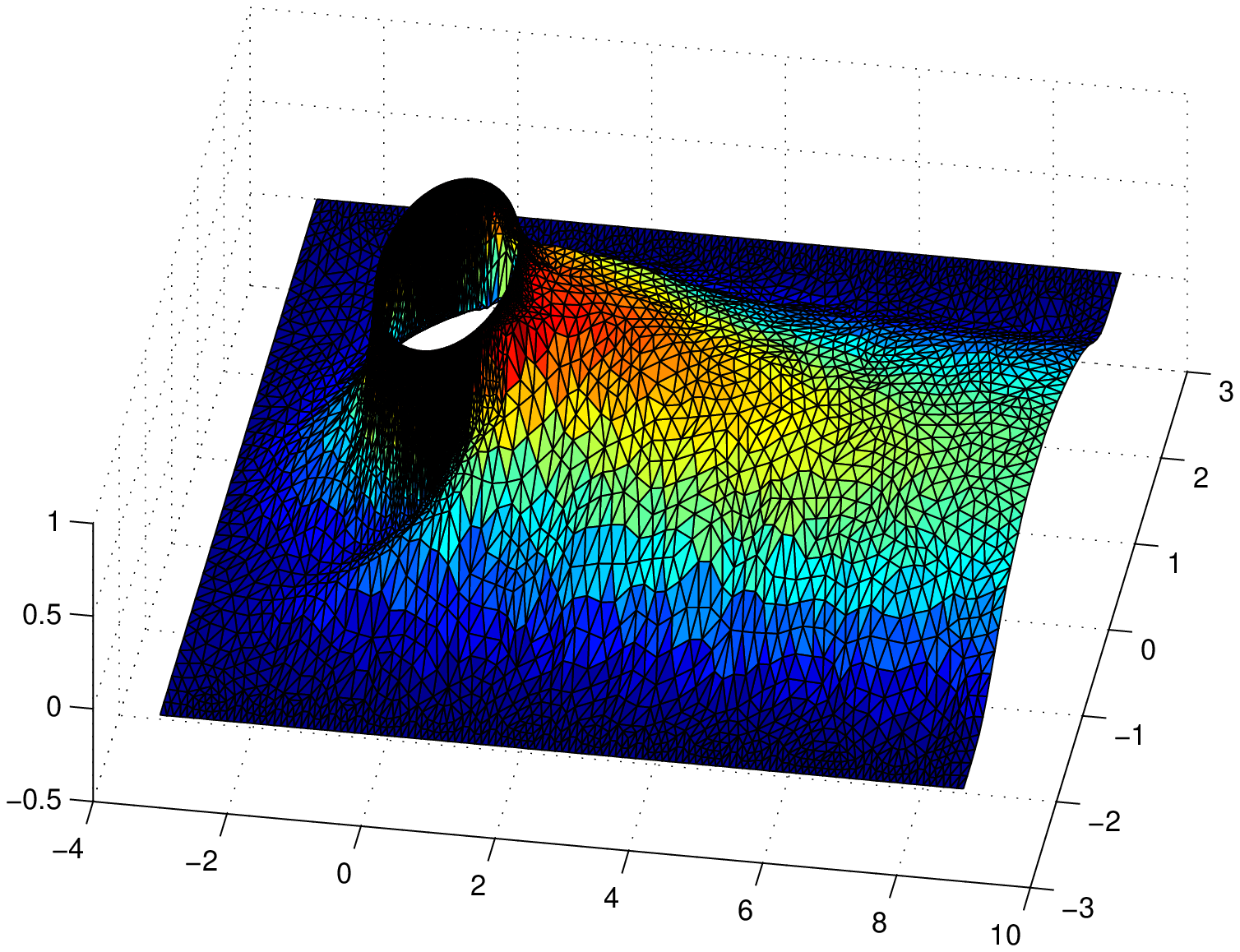}}}
\put(1.27,0.2){\makebox(6,6){\includegraphics[width=3.3cm]{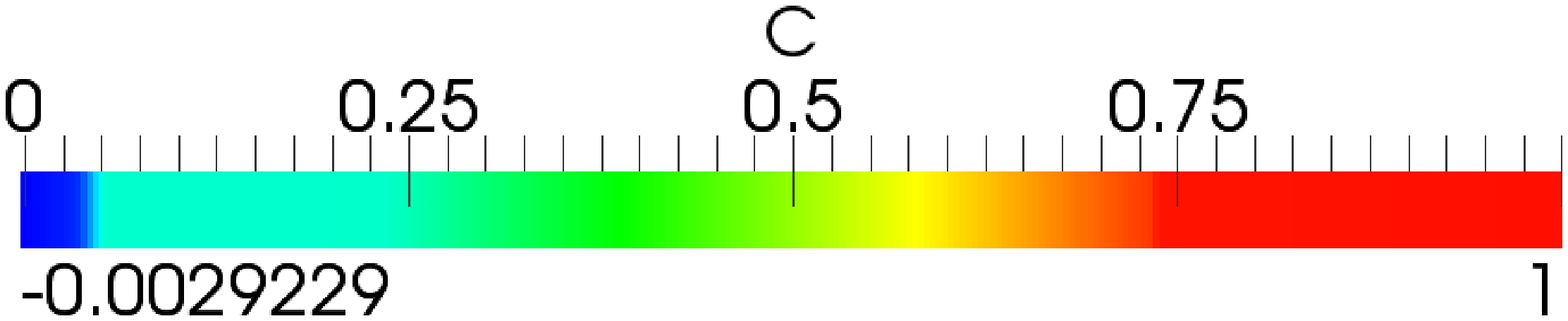}}}
\put(7.5,-1.5){\makebox(6,6){\includegraphics[width=7cm]{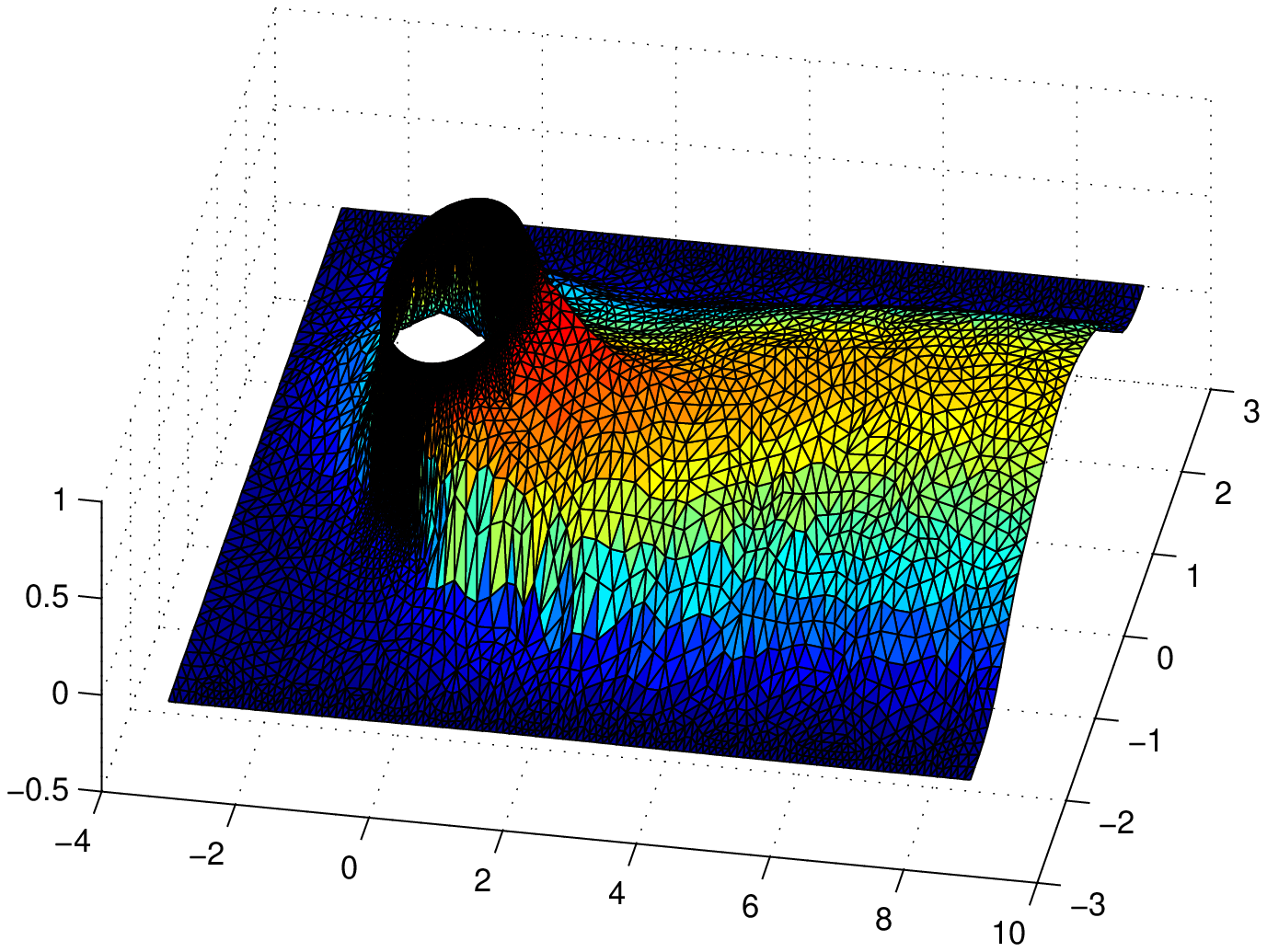}}}
\put(8.73,0.20){\makebox(6,6){\includegraphics[width=3.3cm]{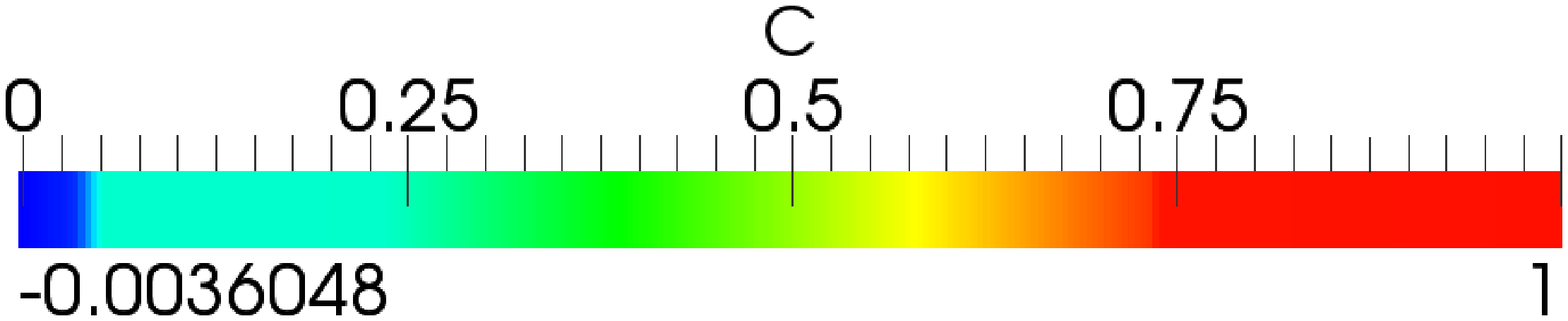}}}
\end{picture}
\end{center}
\caption{The sequence of solutions obtained for the Example 2 with the SUPG $\delta_0 = 0.1$ method for Implicit Euler case at different instance $t = 0.05, 4, 7, 10 $. \label{Ex2SupgEuler}}
\end{figure}

\begin{figure}[t!]
\begin{center}
\unitlength1cm
\begin{picture}(14,11)
\put(0,3.8){\makebox(6,6){\includegraphics[width=7cm]{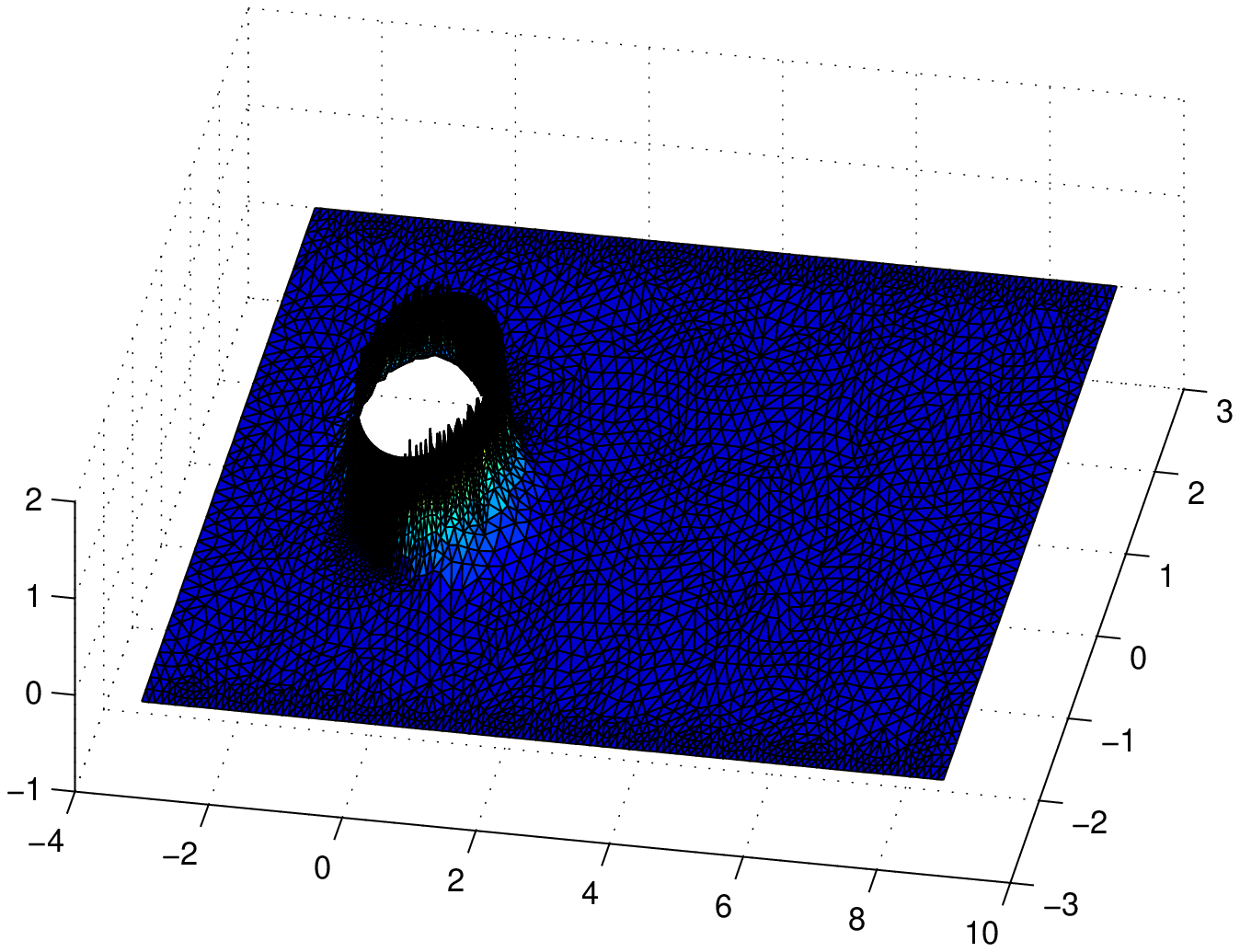}}}
\put(1.2,5.4){\makebox(6,6){\includegraphics[width=3.3cm]{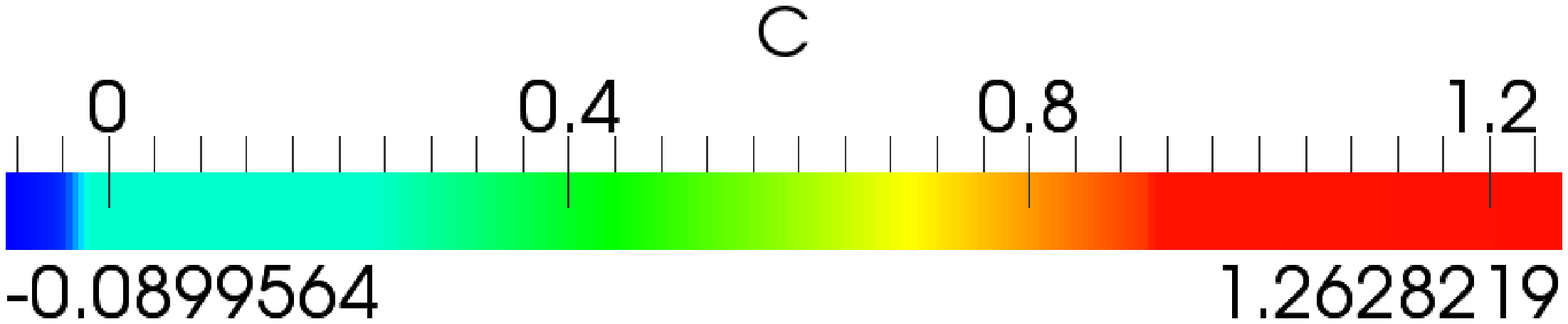}}}
\put(7.5,3.8){\makebox(6,6){\includegraphics[width=7cm]{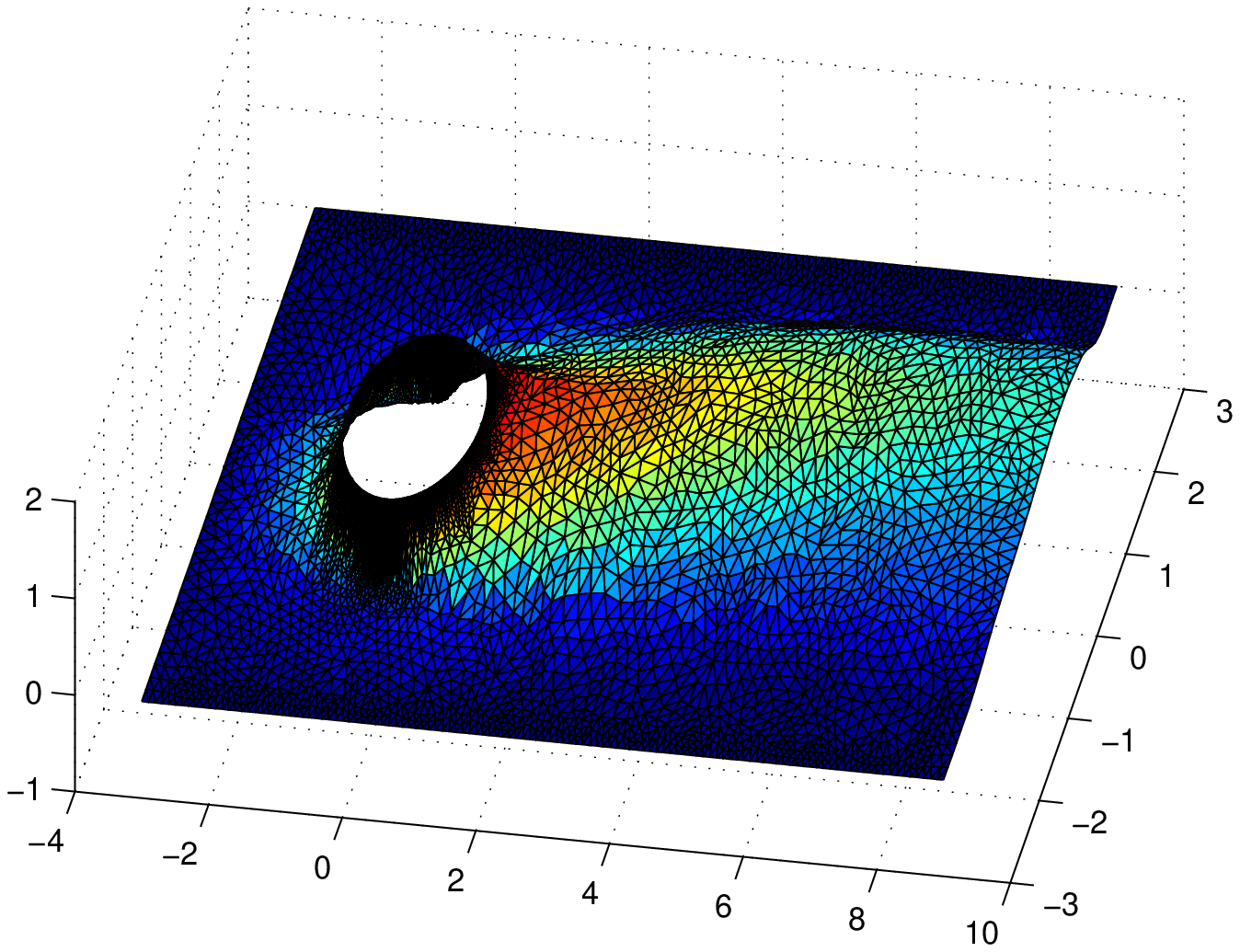}}}
\put(8.71,5.4){\makebox(6,6){\includegraphics[width=3.3cm]{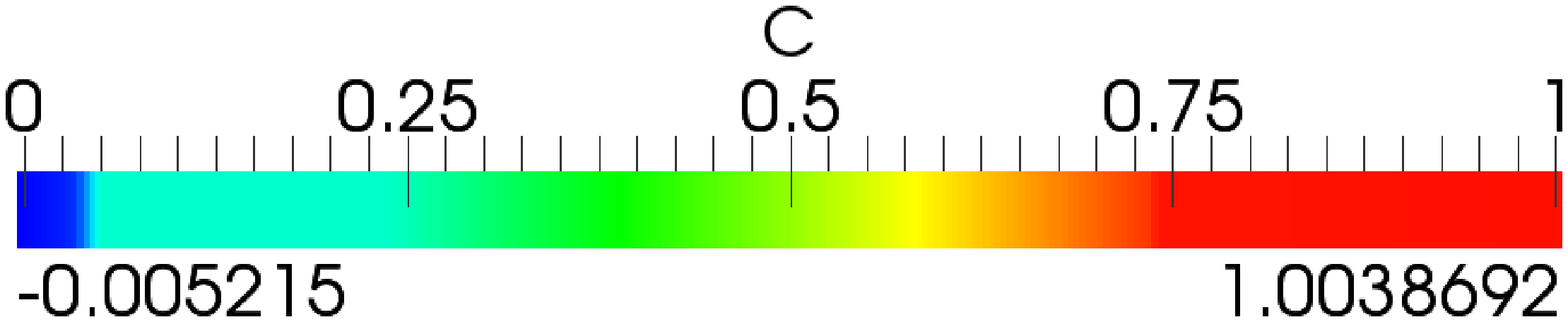}}}
\put(0,-1.5){\makebox(6,7){\includegraphics[width=7cm]{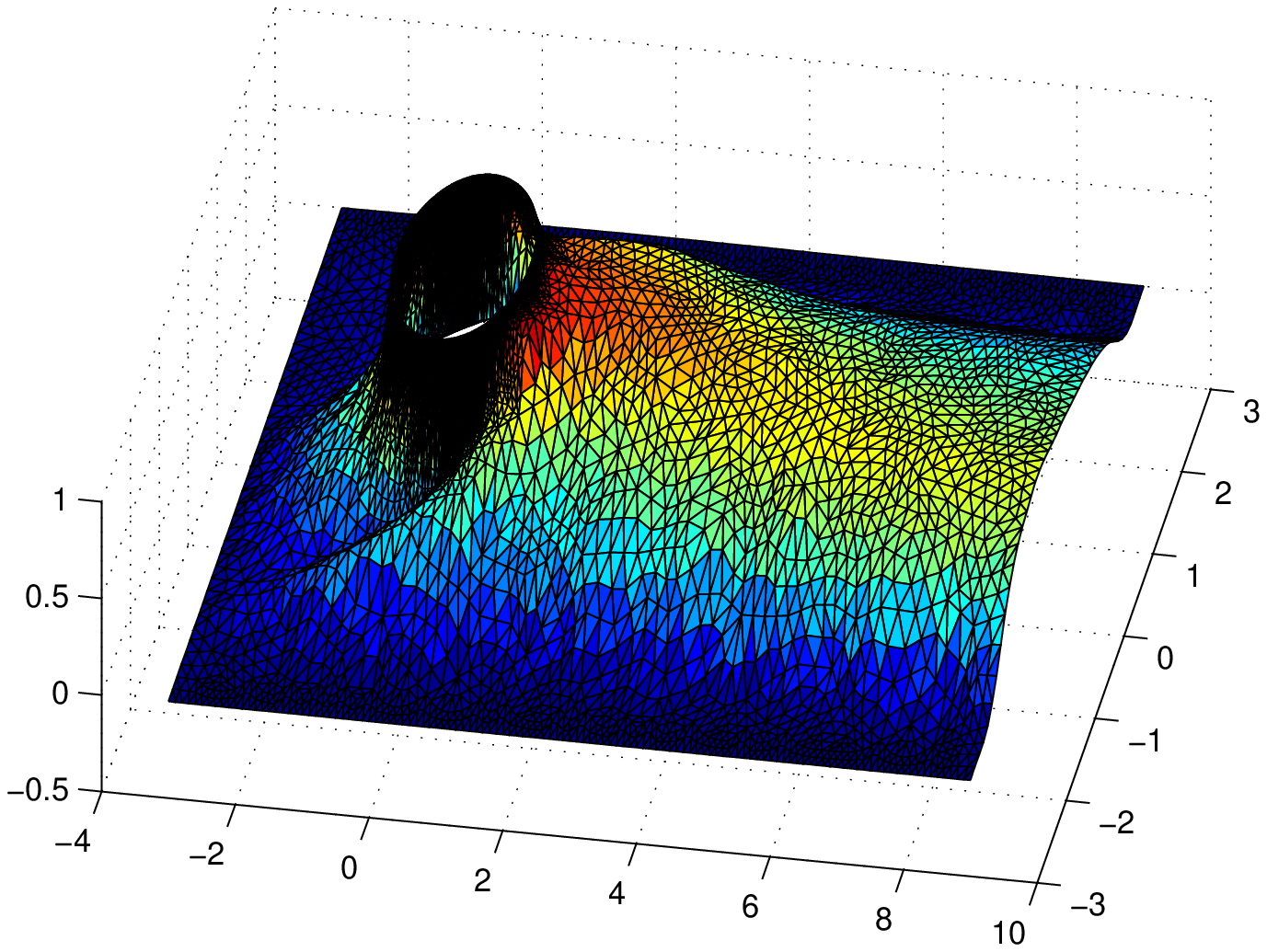}}}
\put(1.25,.60){\makebox(6,6){\includegraphics[width=3.3cm]{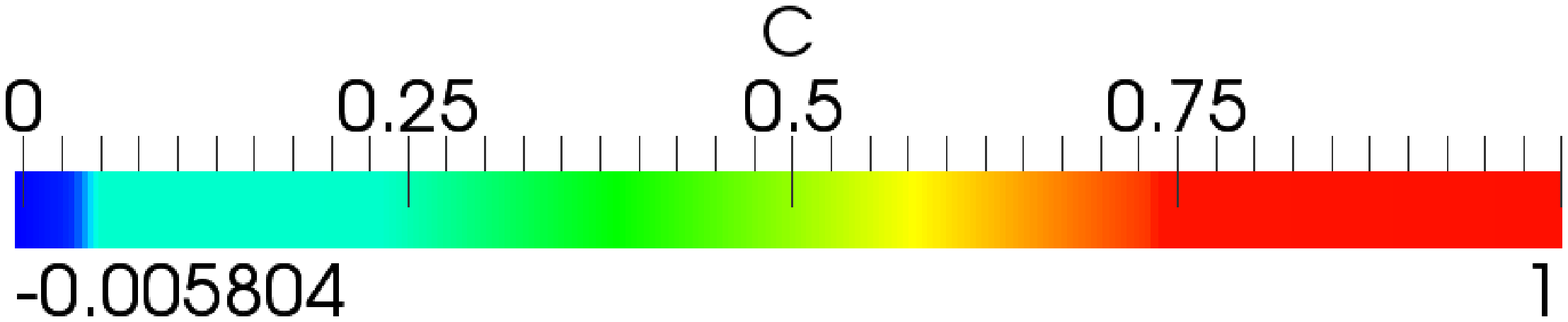}}}
\put(7.5,-1.5){\makebox(6,7){\includegraphics[width=7cm]{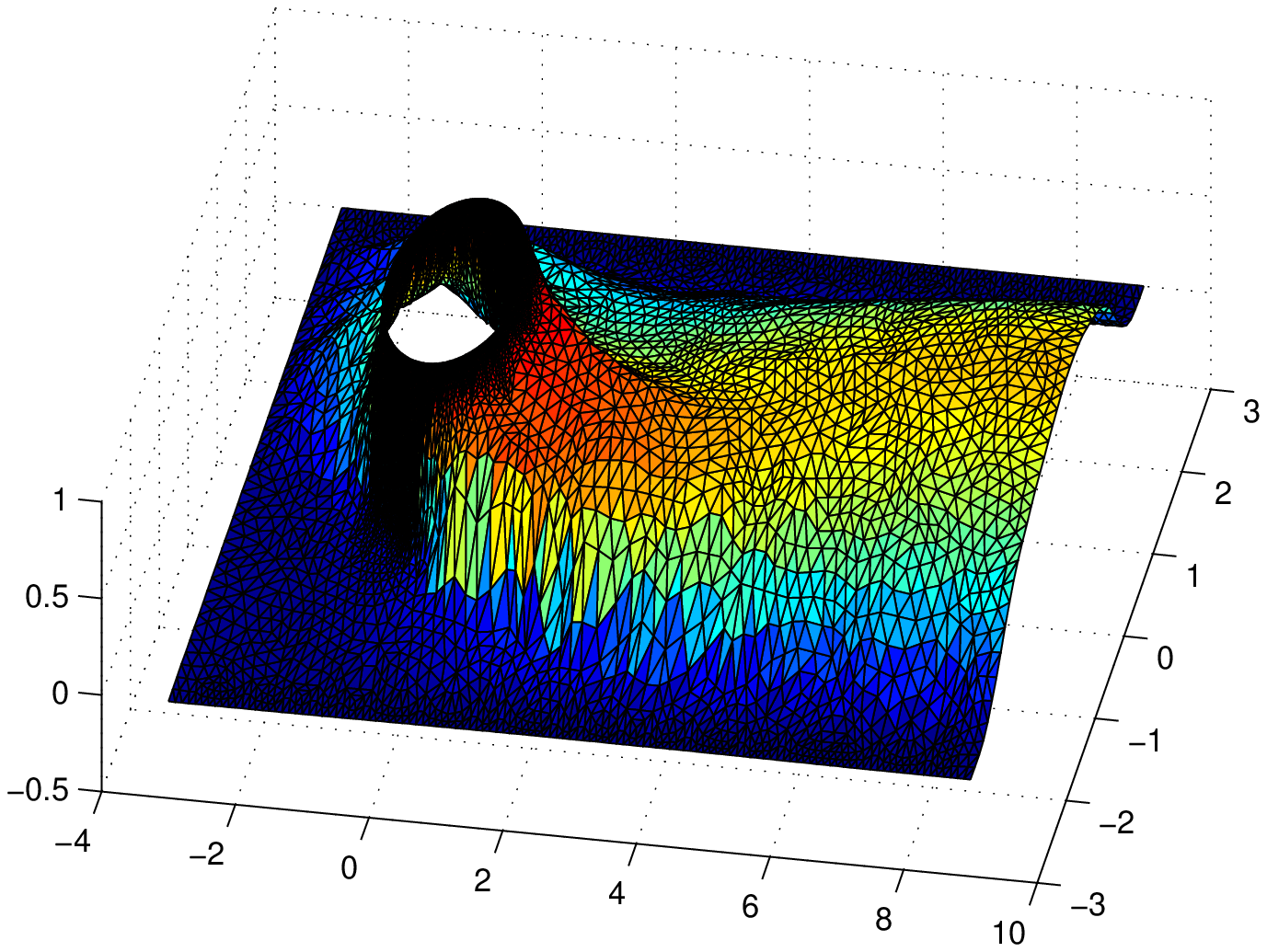}}}
\put(8.7,-.250){\makebox(6,8){\includegraphics[width=3.3cm]{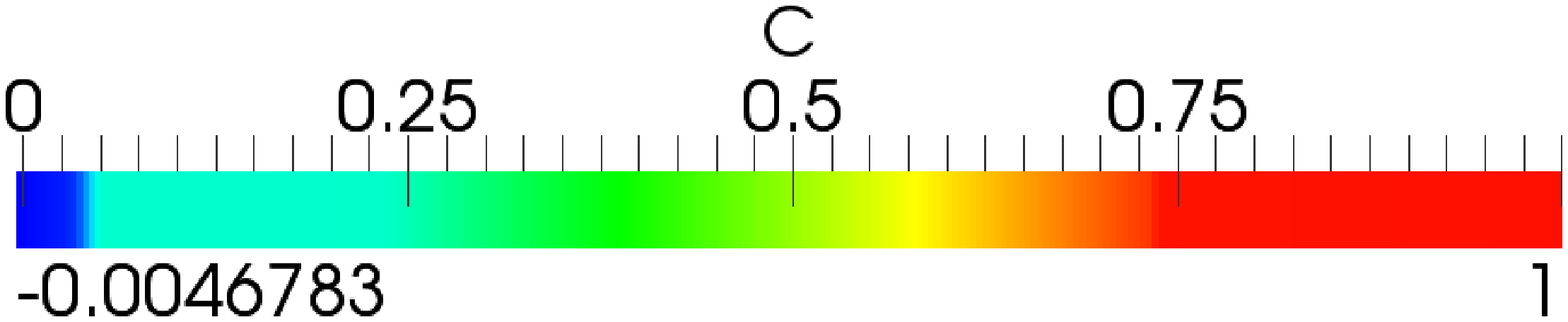}}}
\end{picture}
\end{center}
\caption{The sequence of solutions obtained for the Example 2 with SUPG discretization with $\delta_0 = 0.1$ for Crank-Nicolson method at different instance $t = 0.05, 4, 7, 10$. \label{Ex2SupgCrank}}
\end{figure}

 \section{Summary} In this work, a stabilized numerical scheme for  a transient scalar equation in a time-dependent domain is proposed. In particular, a conservative ALE-SUPG finite element method is analyzed for a convection dominated transient equation in a moving domain. The stability estimates of conservative ALE-SUPG finite element method with the backward Euler and Crank-Nicolson temporal discretizations are derived. 
 The SUPG finite element solution coincides with the standard Galerkin solution for a small value of the stabilization parameter when the convection term is zero or not dominant. 
 The main purpose of the proposed numerical scheme is to approximate the solution of a convection dominant equation in a time-dependent domain where the standard Galerkin method fails or induce spurious oscillations. The robustness of proposed conservative ALE-SUPG is demonstrated with   appropriate examples.

  \bibliographystyle{elsarticle-num} 
 \bibliography{masterlit}

\begin{thebibliography}{10}
\expandafter\ifx\csname url\endcsname\relax
  \def\url#1{\texttt{#1}}\fi
\expandafter\ifx\csname urlprefix\endcsname\relax\def\urlprefix{URL }\fi
\expandafter\ifx\csname href\endcsname\relax
  \def\href#1#2{#2} \def\path#1{#1}\fi

\bibitem{JAS96}
J.~A. Sethian, Level Set Methods, Cambridge University Press, 1996.

\bibitem{HIR74}
C.~W. Hirt, A.~A. Amsden, J.~L. Cook, An arbitrary {L}agrangian {E}ulerian
  computing method for all flow speeds, J. Comput. Phys. 14~(3) (1974)
  227--253.

\bibitem{PES02}
C.~S. Peskin, The immersed boundary method, Acta Numerica (2002) 1--36.

\bibitem{TRY01}
G.~Tryggvason, B.~Bunner, A.~Esmaeeli, D.~Juric, N.~Al-Rawahi, W.~Tauber,
  J.~Han, S.~Nas, Y.-J. Jan, A front-tracking method for the computations of
  multiphase flow, J. Comput. Phys. 169~(2) (2001) 708--759.

\bibitem{DON83}
J.~Don\'{e}a, Arbitrary {L}agrangian-{E}ulerian finite element methods, in:
  T.~Belytschko, T.~R.~J. Hughes (Eds.), Computational methods for transient
  analysis, Elsevier scientific publishing co., Amsterdam, 1983, pp. 473--516.

\bibitem{BO04}
D.~Boffi, L.~Gastaldi, Stability and geometric conservation laws for {ALE}
  formulations, Comp. Meth. in App. Mech. and Engg 193 (2004) 4717--4739.

\bibitem{NOB01}
F.~Nobile, Numerical approximation of fluid-structure interaction problems with
  application to haemodynamics, Ph{D} thesis, \'{E}cole Polytechnique
  F\'{e}d\'{e}rale de Lausanne (2001).

\bibitem{BH82}
A.~N. Brooks, T.~J.~R. Hughes, Streamline upwind/{P}etrov-{G}alerkin
  formulations for convection dominated flows with particular emphasis on the
  incompressible {N}avier-{S}tokes equations, Comput. Methods Appl. Mech. Eng.
  32 (1982) 199--259.

\bibitem{BUR10}
E.~Burman, Consistent {SUPG}-method for transient transport problems:
  {S}tability and convergence, Comput. Methods in Appl. Mech. and Engrg. 199
  (2010) 1114--1123.

\bibitem{JOH10}
V.~John, J.~Novo, Error analysis of the {SUPG} finite element discretization of
  evolutionary convection-diffusion-reaction equations, SIAM J. Numer. Anal.
  49~(3) (2011) 1149--1176.

\bibitem{GANW10}
S.~Ganesan, An operator-splitting {G}alerkin/{SUPG} finite element method for
  population balance equations: {S}tability and convergence, ESAIM: M2AN 46
  (2012) 1447--1465.

\bibitem{hughes89}
T.~J.~R. Hughes, L.~P. Franca, G.~M. Hulbert, A new finite element formulation
  for computational fluid dynamics: Viii. the {G}alerkin/least--squares method
  for advective--diffusion equations, Comput. Methods Appl. Mech. Eng. 73
  (1989) 173--189.

\bibitem{BUR04}
E.~Burman, P.~Hansbo, Edge stabilization for {G}alerkin approximations of
  convection-diffusion-reaction problems, Comput. Methods Appl. Mech. Eng. 193
  (2004) 1437--1453.

\bibitem{BUR06}
E.~Burman, M.~Fernandez, P.~Hansbo, Continuous interior penalty finite element
  method for {O}seen's equations, SIAM J. Numer. Anal. 44 (2006) 1248--1274.

\bibitem{GAN10}
S.~Ganesan, L.~Tobiska, Stabilization by local projection for
  convection-diffusion and incompressible flow problems, J. Sci. Comput. 43~(3)
  (2010) 326--342.

\bibitem{codina2000}
R.~Codina, Stabilization of incompressibility and convection through orthogonal
  sub--scales in finite element methods, Comput. Methods Appl. Mech. Eng. 190
  (2000) 1579--1599.

\bibitem{ROSS08}
H.-G. Roos, M.~Stynes, L.~Tobiska, Numerical Methods for Singularly Perturbed
  Differential Equations, Springer-Verlag, 2008.

\bibitem{codina98}
R.~Codina, Comparison of some finite element methods for solving the
  diffusion--convection--reaction equation, Comput. Methods Appl. Mech. Eng.
  156 (1998) 185--210.

\bibitem{volker014}
J.~de~Frutos, B.~Garcia-Archilla, V.~John, J.~Novo, An adaptive {SUPG} method
  for evolutionary convection--diffusion equations, Comput. Methods Appl. Mech.
  Eng. 273 (2014) 219--237.

\bibitem{volkerSchmeyer08}
V.~John, E.~Schmeyer, Finite element methods for time--dependent
  convection–-diffusion-–reaction equations with small diffusion, Comput.
  Methods Appl. Mech. Eng. 198 (2008) 475–494.

\bibitem{codina06}
S.~Badia, R.~Codina, Analysis of a stabilized finite element approximation of
  the transient convection--diffusion equation using an {ALE} framework, SIAM
  J. Numer. Anal. 44~(5) (2006) 2159--2197.

\bibitem{Bon013}
A.~Bonito, I.~Kyza, R.~Nochetto, Time-discrete higher order {ALE} formulations:
  A priori error analysis, Numer. Math. 125 (2013) 225--257.

\bibitem{BonKyza013}
A.~Bonito, I.~Kyza, R.~Nochetto, Time-discrete higher order {ALE} formulations:
  Stability, SIAM J. Numer. Anal. 51 (2013) 577--604.

\bibitem{MAC12}
J.~A. Mackenzie, W.~R. Mekwi, An unconditionally stable second--order accurate
  {ALE--FEM} scheme for two--dimensional convection--diffusion problems, IMA J
  Numer Anal. 32~(3) (2012) 888--905.

\end{thebibliography}

%
%
%
\end{document}